\documentclass[11pt,reqno]{amsart}

\usepackage{array}
\usepackage{pdfpages}
\usepackage{oplotsymbl}
\usepackage{todonotes}
\usepackage{yfonts}
\usepackage{hyperref}
\usepackage{pgfplots}
\usepgfplotslibrary{groupplots}
\pgfplotsset{compat=1.6}
\pgfplotsset{every axis title/.append style={at={(0.6,1.1)}}}
\usepackage{tcolorbox}
\tcbuselibrary{skins}
\usepackage{algorithm}
\usepackage{algorithmic}
\usepackage{faktor}
\usepackage{array} 
\usepackage{pgf, tikz} 	
\usetikzlibrary{automata, positioning, arrows, decorations.pathreplacing, decorations.pathmorphing, cd}

\definecolor{brightlavender}{rgb}{0.75, 0.58, 0.89}
\definecolor{amethyst}{rgb}{0.6, 0.4, 0.8} 	
\definecolor{blue-violet}{rgb}{0.54, 0.17, 0.89} 	

\definecolor{burgundy}{rgb}{0.5, 0.0, 0.13} 	
\definecolor{burntorange}{rgb}{0.93, 0.53, 0.18} 	
\definecolor{earthyellow}{rgb}{0.88, 0.66, 0.37}
\definecolor{applegreen}{rgb}{0.55, 0.71, 0.0}	
\definecolor{antiquefuchsia}{rgb}{0.57, 0.36, 0.51}
\definecolor{ultramarine}{rgb}{0.07, 0.04, 0.56}

\usepackage{bbm}

\usepackage[utf8]{inputenc}

\usepackage{graphicx}
\usepackage{epstopdf}

\usepackage{hyperref}
\hypersetup{
    colorlinks = true,
    citecolor=blue,
    filecolor=black,
    linkcolor=blue,
    urlcolor=black
}

\usepackage{caption}
\usepackage{subcaption}
\usepackage{tikz-network}
 	
\usepackage{amsthm}
\usepackage{amssymb}
\usepackage{xfrac}
\usepackage{latexsym}
\usepackage{amsbsy}
\usepackage{amsfonts}
\usepackage{appendix}
\usepackage{pstricks}
\usepackage[utf8]{inputenc}
\usepackage{graphicx}
\usepackage{enumitem}% Pour interrompre et reprendre des énumérations

\numberwithin{equation}{section}
\newtheorem{theorem}{Theorem}[section]
\newtheorem{proposition}[theorem]{Proposition}
\newtheorem{lemma}[theorem]{Lemma}

\newtheorem{definition}[theorem]{Definition}

\newtheorem{notation}[theorem]{Notation}

\usepackage[backend=bibtex, doi=false, isbn=false, url=false, backend=biber, style=alphabetic, sorting=ynt]{biblatex}

\begin{document}

\title[Cycle Double Cover]{The cycle double cover conjecture from the perspective of percolation theory on iterated line graphs}

\author[J. W. Fischer]{\textbf{\quad {Jens Walter} Fischer}}
\address{{\bf {Jens Walter} FISCHER}\\ Switzerland.} \email{jefischer@posteo.de}

\begin{abstract}
	The cycle double cover conjecture is a long standing problem in graph theory, which links local properties, the valency of a vertex and no bridges, and a global property of the graph, being covered by a particular set of cycles. We prove the conjecture using a lift of walks and cycles in $G$ to sets of open and closed edges on $\mathcal{L}(\mathcal{L}(G))$, the line graph of the line graph of $G$. We exploit that triangles are preserved by the line graph operator to obtain a one-to-one mapping from walks in the underlying graph $G$ to walks on $\mathcal{L}(\mathcal{L}(G))$. We prove that each set of "double walk covers" in $G$ induces a certain set of $\lbrace 0,1\rbrace$ labels on a subgraph covering of $\mathcal{L}(\mathcal{L}(G))$, minus a set of triangles, and conversely, that there is such a set of labels such that its projection back to $G$ implies a double cycle cover, if $G$ is an simple bridgeless triangle-free cubic graph. The techniques applied are inspired by percolation theory, flipping the $\lbrace 0,1\rbrace$ labels to obtain the desired structure.
\end{abstract}
\bigskip

\maketitle

\textit{ Key words : Cycle Double Cover Conjecture, Line Graph, Combinatorics}  
\bigskip

\textit{ MSC 2020 : 05C10, 57M15}

\section{Introduction \& main statements}
	We consider in what follows finite simple connected graphs $G=(V,E)$, with a particular focus on the edge set $E$. Line graphs have been around for quite some time, Whitney proved already in 1932 in \cite{Whi32} that (almost) any graph can be reconstructed from its line graph. Additionally, a graph can be identified to be the line graph of some other graph in linear time, see \cite{Rous73} and \cite{Lehot74}. The direct extension to iterated line graphs, meaning the line graph of the line graph or the line graph of the line graph of the line graph and so on, has also been focus of research, for example in the context of Hamiltonian line graphs, see \cite{Char68}, or more recent on their spectra in the context of regular graphs, see \cite{RAM05}. The inverse direction, to conclude from properties of an iterated line graph structural properties of the underlying graph $G$, seems to be rarer in occurrence. We want to go this path to tackle an open problem, which requires a global view on the adjacency structure of $G$.\par  
	Our problem at hand is the Cycle Double Cover Conjecture, a long standing problem in graph theory dating back to Szekeres in 1973, see \cite{Sze73}, and Seymour in 1980, see \cite{Sey80}. It is closely linked to the genus of a graph and can be formulated for bridgeless cubic graphs, covering the remaining graph classes as well due to edge-contraction arguments, which preserve the cycle double cover, which is defined as follows.
	\begin{definition}\label{def:cycle_double_cover}
		Let $G=(V,E)$ be a simple connected graph and a set of cycles in $G$ denoted by $\mathcal{C}$ such that for all $e\in E$, there are distinct $\gamma_1,\gamma_2\in\mathcal{C}$ with $e\in\gamma_1$ and $e\in\gamma_2$. Then, the set $\mathcal{C}$ is called a cycle double cover of $G$.
	\end{definition}
	The graph has to be bridgeless, since otherwise the bridge has to be traversed twice by the same cycle preventing a cycle double cover in the sense of Definition \ref{def:cycle_double_cover}. Additionally, we can assume that the underlying graph is triangle free, since expansions of a vertex of valency $3$ into a triangle also preserves the Cycle Double Cover. Consequently, we can formulate it as the following theorem.
	\begin{theorem}\label{thm:cycle_double_cover}
		Let $G=(V,E)$ be a connected simple bridgeless triangle-free cubic graph. Then, there is a cycle double cover such that no edge is covered twice by the same cycle.
	\end{theorem}
	On a high level, we can outline our proof as follows. For a given connected simple bridgeless triangle-free cubic graph $G=(V,E)$ we consider an arbitrary set of walks $\mathcal{W}$ in $G$ such that for each edge $e\in E$ it is traversed twice by the walks in $\mathcal{W}$, either by two distinct walks or twice by the same walk. We aim at showing that there is a set of transformations, which transform $\mathcal{W}$ into a set of cycles $\mathcal{C}$, while preserving the traversal property of the edges in $G$. We tackle this by considering edge-cuts, which allow "separation" of walks with repeated edges into two or more cycles. The main obstacle is a consistent view on the edge relations and edge induced subgraph structure. Hence, we lift the problem to context of line graphs, which represent the graph theoretic structure with focus on the edges in $G$. 
	\par
	We need the following central definition for the discussion. 
	\begin{definition}	 
		Let $G=(V,E)$ be a connected simple graph. Then, the line graph $L$ of $G$ is the graph defined by $L=(E,\mathcal{E})$ with $\langle e,\bar{e}\rangle \in \mathcal{E}$ if and only if $e=\langle v,w \rangle,\bar{e}=\langle w,u \rangle$ with $v,w,u\in V, v\neq w\neq u$. 
	\end{definition}
	Geometrically, the line graph lifts the neighborhood relationship of the edges in $G$ into a new graph which potentially disentangles edge related properties of the underlying graph $G$.
	\begin{notation}
		For any connected simple undirected graph $G$ with line graph $L$, we write as the line graph operator $\mathcal{L}(G):=L$.
	\end{notation}
	Furthermore, we are going to work extensively with deletions of edge sets of subgraphs which we define as follows.
	\begin{definition}
		For any connected simple undirected graph $G=(V,E)$ and a subgraph $H=(V',E')$ with $V'\subseteq V$ and $E'\subseteq E$, we define the difference of $G$ and $H$ to be the graph $G\setminus H := (V,E\setminus E')$. 
	\end{definition}	
	In particular, triangles will play an essential role, due to being fixed subgraph structures under the line graph operator as explained in Lemma \ref{lem:line_graph_op_tri_to_tri} 
	\begin{lemma}\label{lem:line_graph_op_tri_to_tri}
		Let $\mathcal{T}:=\lbrace T_i\rbrace$ the set of all triangles in $G$. Then, the set $\mathcal{T}_{\mathcal{L}(G)}=\{\mathcal{L}(T_i)\rbrace$ contains exclusively subgraphs of $\mathcal{L}(G)$ which are triangles. 
	\end{lemma}
	In short, the operator $\mathcal{L}$ maps triangles to triangles. This is however not a one-to-one correspondence because vertices of valency $3$ also introduce a triangle in the line graph. We discuss in Section \ref{sec:construction_sec_order_line_graph_red} that removing a specific set of triangles from $\mathcal{L}(\mathcal{L}(G))$ allows to disentangle cycle structures in $G$. To this end, we fix the following central object of interest. 
	\begin{definition}
		Define the graph $\mathfrak{L}_2(G)$ associated to $G$ as 
		\begin{equation}
				\mathfrak{L}_2(G):=\mathcal{L}(\mathcal{L}(G))\setminus \mathcal{T}_{\mathcal{L}(\mathcal{L}(G))}.
		\end{equation}
		We call it the reduced order two line graph.
	\end{definition}
	We are now going to establish structural results on $\mathfrak{L}_2(G)$\footnote{It turns out that there is an approach to the problem on the level of $G$ and the set of walks $\mathcal{W}$, using a construction based on half-edges. We discuss this in Appendix \ref{app:half_edge_alternative} but the resulting structure is equivalent to $\mathfrak{L}_2(G)$ and the proof method, therefore, the same.}, which we will use extensively in the proof of Theorem \ref{thm:cycle_double_cover}. The proof of Theorem \ref{thm:cycle_double_cover} follows in Section \ref{sec:proof_main_thm} after the presentation of all results on the necessary construction discussed in Sections \ref{sec:construction_sec_order_line_graph_red}  and \ref{sec:proof_prop}. In what follows, let $G=(V,E)$ be a connected simple undirected bridgeless triangle-free cubic graph. 
\section{Structure of reduced order two line graph}\label{sec:construction_sec_order_line_graph_red}
	Recall that we defined the reduced order two line graph as $\mathfrak{L}_2(G)$ with
	\begin{equation*}
		\mathfrak{L}_2(G):=\mathcal{L}(\mathcal{L}(G))\setminus \mathcal{T}_{\mathcal{L}(\mathcal{L}(G))}.
	\end{equation*}
	Our particular interest lies in the remainders of $K_4$ cliques in $\mathcal{L}(\mathcal{L}(G))$ after deletion of $\mathcal{T}_{\mathcal{L}(\mathcal{L}(G))}$. They will play a central role, allowing together with labels on the remaining edges to obtain a "separation of cycles" in $\mathfrak{L}_2(G)$. The set of labels can then be projected to $G$ to obtain a double cycle cover proving Theorem \ref{thm:cycle_double_cover}. We develop the concept in the following subsections.
	\subsection{Reduced order two line graph and reduced cliques}
	In this subsection, we will analyze the structure of $\mathfrak{L}_2(G)$ and, in particular, the role of $K_4$ cliques in $\mathcal{L}(\mathcal{L}(G))$ and their induced subgraphs after deletion of $\mathcal{T}_{\mathcal{L}(\mathcal{L}(G))}$.
	\begin{proposition}\label{prop:reduced_double_line_graph_4_reg_con}
		The graph $\mathfrak{L}_2(G)$ is $4$-regular and connected.
	\end{proposition}
	\begin{proof}
		Note first, that $\mathcal{L}(\mathcal{L}(G))$ is $6$-regular and any vertex in $\mathcal{L}(\mathcal{L}(G))$ belongs to exactly one triangle induced by a triangle in $\mathcal{L}(G)$, i.e., which belongs to $ \mathcal{T}_{\mathcal{L}(\mathcal{L}(G))}$. Consequently, by removing $\mathcal{T}_{\mathcal{L}(\mathcal{L}(G))}$, the valency is reduced by $2$ for any vertex such that $\mathfrak{L}_2(G)$ is $4$-regular. \par
		Furthermore, the graph $\mathcal{L}(\mathcal{L}(G))$ can be seen as a composition of $K_4$ where each vertex belongs to exactly $2$ cliques $K_4$ and any neighbor of a vertex belongs to one of the associated cliques. Let $T^{(2)}$ a triangle in $\mathcal{T}_{\mathcal{L}(\mathcal{L}(G))}$ and $v\in T^{(2)}$. Denote by $\lbrace v_1^{(a)},v_2^{(a)},v_3^{(a)},v\rbrace$ and $\lbrace v_1^{(b)},v_2^{(b)},v_3^{(b)},v\rbrace$ the vertex sets of the two cliques associated to $v$. Additionally, remark that since $v$ belongs to $T^{(2)}$ so does one further vertex of each clique. Without loss of generality, we assume that $T^{(2)}=\lbrace v,v_1^{(a)},v_1^{(b)}\rbrace$. Removing the edges of $T^{(2)}$ from $\mathcal{L}(\mathcal{L}(G))$ removes, consequently, two edges from each clique $K_4$ such that their reduced form stays connected and so does, finally, their composition.  
	\end{proof}
	The central part of the further construction are the subgraphs of $\mathfrak{L}_2(G)$ which arise from the deletion of edges in cliques $K_4$ in $\mathcal{L}(\mathcal{L}(G))$ of the induced triangles from $\mathcal{L}(G)$. In this sense, we call them simply reduced cliques.
	\begin{notation}
		We denote by $\mathcal{X}:=\lbrace \mathbb{X}_i\rbrace_{i=1}^{|E|}$ the reduced cliques $K_4$ constructed in the proof of Proposition \ref{prop:reduced_double_line_graph_4_reg_con} and write $(\mathfrak{L}_2(G),\mathcal{X})$ for the full graph-subgraph structure.  
		\begin{figure}[H]
			\centering
			\includegraphics[scale=1.5]{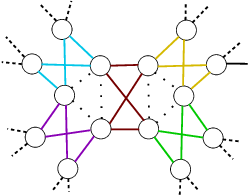}
			\caption{Local representation of $\mathfrak{L}_2(G)$ using colors to represent the reduced cliques $\mathbb{X}$. Dashed lines represent edges leading to the remaining parts of the graph. Dotted lines represent the removed triangles.}	
			\label{fig:local_representation_X}
		\end{figure}
	\end{notation}
	The name of the reduced cliques is motivated by their form in a local representation as is illustrated in Figure \ref{fig:local_representation_X}. They form an edge-disjoint cover of $\mathfrak{L}_2(G)$ , i.e., $\mathbb{X},\mathbb{Y}\in\mathcal{X}$ do not share any edges and the union over all $\mathbb{X}$ gives all of $\mathfrak{L}_2(G)$. Moreover, by construction, any two $\mathbb{X},\mathbb{Y}\in\mathcal{X}$ have at most one vertex in common. In what follows, we assign labels in $\lbrace 0,1\rbrace$ to all edges in $\mathfrak{L}_2(G)$ and by canonical extension to any $\mathbb{X}\in\mathcal{X}$, where $0$ means closed and $1$ means open, in accordance with classical notations in percolation theory. 
	\begin{definition}\label{def:open_closed_edges_in_reduced_double_line}
		Consider the graph-subgraph structure $(\mathfrak{L}_2(G),\mathcal{X})$ together with a set of labels 
		\begin{equation}
			\Lambda:=\left\lbrace  \lbrace 0,1\rbrace^{\lbrace e\in \mathbb{X}\rbrace}\,\big|\,\mathbb{X}\in\mathcal{X}\right\rbrace.
		\end{equation}
		We denote by $\lambda_e\in\lbrace 0,1\rbrace$ the label of the edge $e\in \mathbb{X}$.\par				
		We say that $\Lambda$ is a valid labeling if for all $\mathbb{X}\in\mathcal{X}$ we have for all $v\in\mathbb{X}$ two vertices $w,u\in\mathbb{X}\setminus\lbrace  v\rbrace$ with $\langle v,w\rangle,\langle v,u\rangle\in \mathbb{X}$ such that
		\begin{equation}
			\lambda_{\langle v,w\rangle} = 1-\lambda_{\langle v,u\rangle}.
		\end{equation}
		If $\Lambda$ is a valid labeling, we write $(\mathfrak{L}_2(G),\mathcal{X},\Lambda)$ for the labeled graph-subgraph structure.
	\end{definition} 
	Note that there are effectively two choices of edge-labels for any reduced clique $\mathbb{X}$ to obtain a valid labeling $\Lambda$. A vertex $v\in (\mathfrak{L}_2(G),\mathcal{X},\Lambda)$ is incident to exactly two edges $e_1,e_2$ with $\lambda_{e_1}=\lambda_{e_2}=1$ and edges $e_3,e_4$ with $\lambda_{e_3}=\lambda_{e_4}=0$ in $(\mathfrak{L}_2(G),\mathcal{X},\Lambda)$ since $v$ belongs to exactly two reduced cliques.
	\begin{lemma}\label{lem:open_edges_disjoint_union_of_cycles}
		Given $(\mathfrak{L}_2(G),\mathcal{X},\Lambda)$ the edge-induced subgraph of $\mathfrak{L}_2(G)$ on the set of open edges
		\begin{equation}
			\mathfrak{E}_2^{(open)}:=\lbrace e\in \mathfrak{L}_2(G)\,|\, \lambda_e=1\rbrace		
		\end{equation}				
		is a disjoint union of cycles.
	\end{lemma}
	\begin{proof}
		Any vertex in the edge-induced subgraph has degree $2$ independently of the choice of labels as long as $\Lambda$ is valid. Consequently, any connected component of the edge-induced subgraph is a simple graph with vertices of valency $2$. Therefore, any connected component is a cycle. Disjointness follows as well.  
	\end{proof}
	The edge-induced subgraph of $\mathfrak{L}_2(G)$ on $\mathfrak{E}_2^{(open)}$ will play a central role. Its constituent cycles will give for a specific valid labeling $\Lambda$ after projection to $G$ a double cycle cover. In Figure \ref{fig:local_representation_X_edge_label_induced_subgraph} a local representation of the situation is illustrated. 
	\begin{figure}[H]
		\centering
		\includegraphics[scale=1.5]{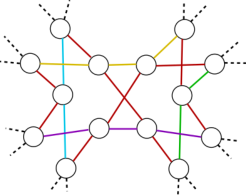}
		\caption{Local representation of $(\mathfrak{L}_2(G),\mathcal{X},\Lambda)$ using colors to represent the closed edges and the open edges inducing cycles. Closed edges are dark red and open edges are the remaining colors. Open edges of identical color belong to the same cycle.}	
		\label{fig:local_representation_X_edge_label_induced_subgraph}
	\end{figure}
	In fact, for an underlying graph $G=(V,E)$, working combinatorically through all valid labels of the reduced cliques is unfeasible since there are $2^{|E|}$ possibilities. It turns out that considering $(\mathfrak{L}_2(G),\mathcal{X},\Lambda)$ and, in particular, a valid labeling $\Lambda$, it is more insightful to look globally at the cycles induced by $\mathfrak{E}_2^{(open)}$ on $\mathfrak{L}_2(G)$.
		\begin{notation}
			Given $(\mathfrak{L}_2(G),\mathcal{X},\Lambda)$, we denote by $\Gamma_{\Lambda}$ the set of disjoint cycles resulting from Lemma \ref{lem:open_edges_disjoint_union_of_cycles}. A cycle will be denoted by $\gamma\in\Gamma_{\Lambda}$.
		\end{notation}
		Given $(\mathfrak{L}_2(G),\mathcal{X},\Lambda)$ we can make a link to the line graph $\mathcal{L}(G)$ as follows. First, embedd canonically $\Gamma_{\Lambda}$ into $\mathcal{L}(\mathcal{L}(G))$, then assign via projection to any edge in $\mathcal{L}(G)$ a color and two edges in $\mathcal{L}(G)$ have the same color if and only if the corresponding vertices in $\mathfrak{L}_2(G)$ belong to the same $\gamma\in\Gamma_{\Lambda}$. To make the map well-defined, we have to restrict the colorings of trails in $\mathcal{L}(G)$ since not all give a valid labeling of $\mathfrak{L}_2(G)$. 
		\begin{definition}\label{def:valid_colorings_line_graph}
			Let $\mathcal{T}$ be a set of closed trails in $\mathcal{L}(G)$ such that all vertices of $\mathcal{L}(G)$ are contained in some trail $\tau$ in $\mathcal{T}$. Denote by $\Delta=\lbrace \Delta_1,\Delta_2,\Delta_3\rbrace$ a triangle in $\mathcal{L}(G)$, which is induced by a single vertex in $G$ (see \cite{Har72}). Denote the remaining two neighbors of $\Delta_1$ by $v_1$ and $v_2$. We say $\tau$ induces a valid edge-labeling if for all trails $\phi\in\tau$ and $\Delta_1\in \phi$ as well as $\langle \Delta_1,\Delta_2\rangle,\langle \Delta_1,\Delta_3\rangle\in \phi$ implies $\langle \Delta_1,v_1\rangle,\langle \Delta_1,v_2\rangle\in \phi$.
		\end{definition}
		We will from now on only consider valid colorings in the sense of Definition \ref{def:valid_colorings_line_graph}.
		\begin{notation}\label{not:chromatic_from_double_to_line_graph}
			We denote by $\chi_{\mathcal{E}}$ the projection from $\Gamma_{\Lambda}\subset \mathcal{L}(\mathcal{L}(G))$ to an edge coloring of $\mathcal{L}(G)$ and the lift from a valid edge-labeling of $\mathcal{L}(G)$ to some $\Gamma_{\Lambda'}$ by $\chi_{\mathcal{E}}^{-1}$.
		\end{notation}	
		Indeed, the projection $\chi_{\mathcal{E}}$ has an important structural property.
		\begin{lemma}
			Given $(\mathfrak{L}_2(G),\mathcal{X},\Lambda)$, the projection $\chi_{\mathcal{E}}(\Gamma_{\Lambda})$ induces a valid edge-labeling on $\mathcal{L}(G)$ and can, in particular, be represented as a union of close trails covering $\mathcal{L}(G)$.
		\end{lemma}
		\begin{proof}
		Since $\Gamma_{\Lambda}$ consists of a disjoint union of cycles, covering all vertices of $\mathcal{L}(\mathcal{L}(G))$ and each vertex in $\mathcal{L}(\mathcal{L}(G))$ has a one-to-one correspondence to an edge in $\mathcal{L}(G)$, all that remains to be shown is that the projection gives indeed closed trails and they avoid covering all sides of a triangle in $\mathcal{L}(G)$. But the projected cycles cannot cover all sides of a triangle in $\mathcal{L}(G)$ because then a $\gamma\in\Gamma_{\Lambda}$ would already contain a triangle in $(\mathfrak{L}_2(G),\mathcal{X},\Lambda)$, the latter being triangle free. The trails are closed since they arise from projected cycles in $\Gamma_{\Lambda}$.
		\end{proof}
	Each valid labeling can now be projected to $G$ as shown in Figure \ref{fig:projection_diagram_walks}, where $\pi$ assigns to any edge in $G$ one or two colors, corresponding to the colors of the edges incident to the corresponding vertex in $\mathcal{L}(G)$.
	\begin{figure}[H]
		\centering
		\begin{tikzcd}
			G  \arrow[r, "\mathcal{L}"] & \mathcal{L}(G) \arrow[r, "\mathcal{L}"] & \mathcal{L}(\mathcal{L}(G)) \arrow[d] \\ 
			\mathcal{W}_{G} \arrow[u, "\text{cover}"]  & \mathcal{T}_{\mathcal{L}(G)} \arrow[u, "\text{cover}"] \arrow[l, "\pi"] & (\mathfrak{L}_2(G),\mathcal{X},\Lambda) \arrow[l, "\chi_{\mathcal{E}}"]\\
		\end{tikzcd}
		\caption{The diagram illustrates first the construction of $(\mathfrak{L}_2(G),\mathcal{X},\Lambda)$. Then, subsequently, it shows the projection via $\chi_{\mathcal{E}}$ to a spanning set of closed trails $\mathcal{T}_{\mathcal{L}(G)}$ in $\mathcal{L}(G)$ and, then, via $\pi$ to a set of closed walks $\mathcal{W}_{G}$ in $G$ such that the union over all walks traverses every edge exactly twice.}
		\label{fig:projection_diagram_walks}
	\end{figure}
	We fix one important property in the following observation, which follows directly from the construction as mentioned in Figure \ref{fig:projection_diagram_walks}.
	\begin{lemma}\label{lem:line_graph_trail_proj}
		Given $(\mathfrak{L}_2(G),\mathcal{X},\Lambda)$ and the set $\mathcal{T}_{\mathcal{L}(G)}:=\chi_{\mathcal{E}}(\Gamma_{\Lambda})$, being a valid edge-labeling of $\mathcal{L}(G)$, its projection $\mathcal{W}_{G}:=\pi(\mathcal{T}_{\mathcal{L}(G)})$ to $G$ is a set of closed walks such that the union over all walks traverses every edge exactly twice.
	\end{lemma}
	\begin{proof}
		We use a proof by drawing to show that every edge is traversed twice. See Figure \ref{fig:projection_line_graph_colors} which depicts the colored line graph locally around some vertex $v$ in $\mathcal{L}(G)$. The coloring of the edges incident to $v$ give canonically two, not necessarily unique, colors of the corresponding edge in $G$. Each color gives rise to a walk in $G$ and, hence, the resulting family of walks traverses each edge exactly twice. 
		\begin{figure}[H]
			\centering
			\includegraphics{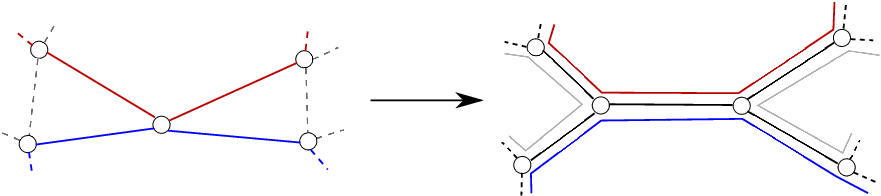}
			\caption{Projection from valid edge-labeling of $\mathcal{L}(G)$ on the left to double coloring of edge of $G$. The arrow indicates the mapping through $\pi$.}	
			\label{fig:projection_line_graph_colors}
		\end{figure}
	\end{proof}
	From the proof of Lemma \ref{lem:line_graph_trail_proj} and, in particular, Figure \ref{fig:projection_line_graph_colors} it is clear, that we need necessarily two unique colors in the neighborhood of each vertex in $\mathcal{L}(G)$, to obtain a cycle double cover of $G$. The corresponding labeling $\Lambda$ is exactly given by the labels which are projected under $\pi\circ\chi_{\mathcal{E}}$ to a set of cycles in $G$. They have a specific structure which we define in the next Subsection \ref{sec:typeA_B_intersections} and finally show in Section \ref{sec:proof_prop} their existence in the case of bridgeless cubic graphs.
	\subsection{Type A and Type B self-intersections}\label{sec:typeA_B_intersections}
		In this subsection, we develop on the structure of the cycles in $\Gamma_{\Lambda}$, given $(\mathfrak{L}_2(G),\mathcal{X},\Lambda)$. For this, we need a few notations and definitions. They should render the construction more accessible since they lift everything to the space of induced cycles. We are going to establish a structural understanding by defining adjacent cycles in $(\mathfrak{L}_2(G),\mathcal{X},\Lambda)$ before coming back to the projection from  $(\mathfrak{L}_2(G),\mathcal{X},\Lambda)$ to $\mathcal{L}(G)$.
		\begin{definition}\label{def:adjacent_cycles}
			Consider $(\mathfrak{L}_2(G),\mathcal{X},\Lambda)$ and let $\gamma_1,\gamma_2\in\Gamma_{\Lambda}$ cycles. If there is a reduced clique $\mathbb{X}\in\mathcal{X}$ and edges $e_1\in\gamma_1$, $e_2\in\gamma_2$ such that $e_1,e_2\in \mathbb{X}$, then we call $\gamma_1,\gamma_2$ adjacent in $(\mathfrak{L}_2(G),\mathcal{X},\Lambda)$. Additionally, we say that $\mathbb{X}$ joins $\gamma_1$ and $\gamma_2$.
		\end{definition}
		Note that two cycles can be joined by multiple $\mathbb{X}\in\mathcal{X}$ but remains a binary property. The overarching structure can, hence, be captured as follows.
		\begin{lemma}\label{lem:graph_of_cycles}
			Given $(\mathfrak{L}_2(G),\mathcal{X},\Lambda)$, the graph $\mathfrak{G}_{\Lambda}:=(\Gamma_{\Lambda},\mathfrak{E}_{\Lambda})$ where $\langle \gamma,\phi\rangle\in \mathfrak{E}_{\Lambda}$ if and only if $\gamma,\phi\in\Gamma_{\Lambda}$ and $\gamma$ and $\phi$ are adjacent $(\mathfrak{L}_2(G),\mathcal{X},\Lambda)$, is a simple connected graph with loops.
		\end{lemma}	
		\begin{proof}
			To start, the graph $\mathfrak{G}_{\Lambda}$ is simple with potential loops, since adjacency may occur between $\gamma\in\Gamma_{\Lambda}$ and itself, implying loops, but is due to structure a binary property, implying that $\mathfrak{G}_{\Lambda}$ is simple. On the other hand, to show that $\mathfrak{G}_{\Lambda}$ is connected, we consider the projection $\chi_{\mathcal{E}}$ from $(\mathfrak{L}_2(G),\mathcal{X},\Lambda)$ to $\mathcal{L}(G)$. Assume that $\mathfrak{G}_{\Lambda}$ has at least two connected components. Since the projection $\chi_{\mathcal{E}}(\Gamma_{\Lambda})$ gives a valid edge-labeling of $\mathcal{G}$, this implies that $\mathcal{L}(G)$ can be separated in two subsets of edges with disjoint colorings $\mathcal{C}_1$ and $\mathcal{C}_2$, with no two edges $e_1\in\mathcal{C}_1$ and $e_2\in\mathcal{C}_2$ being incident to the same vertex. Consequently, $\mathcal{L}(G)$ has at least two connected components, which is a contradiction to $G$ being connected.  
		\end{proof}			
		We come back to the properties of $\mathfrak{G}_{\Lambda}$ in Subsection \ref{subsub:monotonicity_type_a_deletion} where it serves to control the structural evolution of $\Gamma_{\Lambda}$ under transformation of the labels. Adjacency can locally be understood as illustrated in Figure \ref{fig:local_representation_adjacent_cycles}. If two cycles are adjacent, the layout in Figure \ref{fig:local_representation_adjacent_cycles} can naturally be transformed to the diagonals being open. 
		\begin{figure}[H]
			\centering
			\includegraphics[scale=1.5]{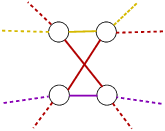}
			\caption{Two adjacent cycles and a joining reduced clique. Red edges are again closed and cycles are yellow and purple.}	
			\label{fig:local_representation_adjacent_cycles}
		\end{figure}
		Nonetheless, for illustration purposes, we will use the layout presented in Figure \ref{fig:local_representation_adjacent_cycles} as an idea for adjacent cycles while the diagonals will be employed for self-intersections as defined in Definition \ref{def:self_intersection_cycles}.
		\begin{definition}\label{def:self_intersection_cycles}
			Let $\gamma\in\Gamma_{\Lambda}$. If there is a reduced clique $\mathbb{X}$ such that both open edges in $\mathbb{X}$ belong to $\gamma$, then we say that $\gamma$ has a self-intersection and $\mathbb{X}$ is a self-intersection of $\gamma$. 
		\end{definition}
		Again, self-intersections are not necessarily unique. A possible illustration shown in Figure \ref{fig:local_representation_self_int_cycles} serves, again, as an idea or intuition. In particular, self-intersections will play a crucial role in the proof in Section \ref{sec:proof_main_thm}. They may arise from families of walks in $G$, such that any edge in $G$ is traversed twice by walks in said family. 
		\begin{figure}[H]
			\centering
			\includegraphics[scale=1.5]{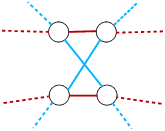}
			\caption{A self-intersection $\mathbb{X}$ of some cycle $\gamma$. The cycle $\gamma$ is colored in light blue.}	
			\label{fig:local_representation_self_int_cycles}
		\end{figure}
		The distinction between adjacency and self-intersections is a binary property of reduced cliques. Moreover, self-intersections of any cycle $\gamma$ can, again, be differentiated into two topologically different types. For this, we need the following cycle notation, which is inspired by knot theory.
		\begin{definition}\label{not:cycle_types}
			Let $(\mathfrak{L}_2(G),\mathcal{X},\Lambda)$ a reduced order two line graph with reduced cliques $\mathcal{X}$ and valid labeling $\Lambda$. Then we distinguish four types of cycles in $\Gamma_{\Lambda}$ using the following set of diagrams.\footnote{Those are similar to knot diagrams, being not unique for a given cycle, with an additional local-global component in cases b)-d).} Black lines in a reduced clique $\mathbb{X}\in\mathcal{X}$ are open edges, red ones closed.
			\begin{enumerate}[label=\alph*)]
				\item A simple cycle without any local structure exposed:\vspace{5pt}\par
							\begin{minipage}{\linewidth}
								\centering
								\includegraphics[scale=0.6]{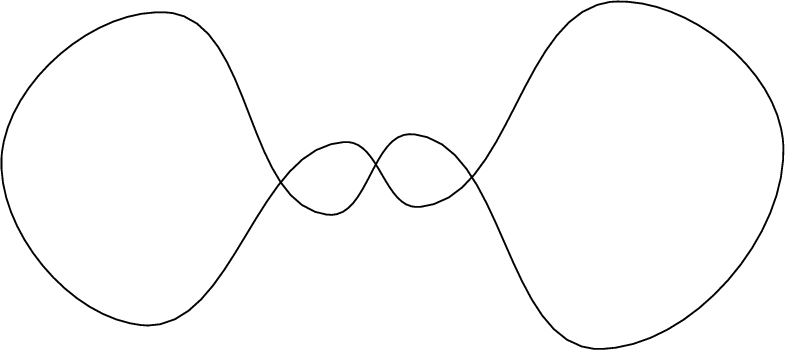}
							\end{minipage}\vspace{7pt}
%				\item A simple cycle without any self-intersections:\vspace{5pt}\par
%							\begin{minipage}{\linewidth}
%								\centering
%								\includegraphics[scale=0.5]{images/simple_cycle}
%							\end{minipage}\vspace{7pt}
%				\item Two adjacent cycles joined by some reduced clique:\vspace{5pt}\par
%							\begin{minipage}{\linewidth}
%								\centering
%								\includegraphics[scale=1.5]{images/local_rep_diagram_adjacency_reduced_clique}
%							\end{minipage}\vspace{7pt}
				\item A self-intersecting cycle with focus on a self-intersection $\mathbb{X}$, called \textbf{Type A}:\label{item:self_int_Type_A}\vspace{5pt}\par
							\begin{minipage}{\linewidth}
								\centering
								\includegraphics[scale=1,angle=90]{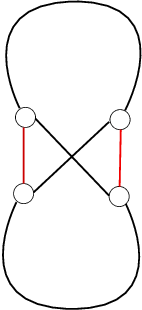}
							\end{minipage}\vspace{7pt}
				\item A self-intersecting cycle with focus on a self-intersection $\mathbb{X}$, called \textbf{Type B}:\label{item:self_int_Type_B}\vspace{5pt}\par
							\begin{minipage}{\linewidth}
								\centering
								\includegraphics[scale=1]{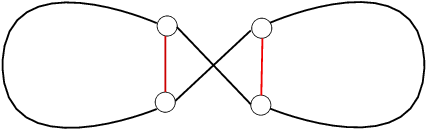}
							\end{minipage}
			\end{enumerate}
		\end{definition}	
		We want to emphasize that "intersections" in the sense of Definition \ref{def:self_intersection_cycles} and as presented in Definition \ref{not:cycle_types} \ref{item:self_int_Type_A} and \ref{item:self_int_Type_B} are not actual intersections which traverse a vertex or an edge twice. It is solely a view on the structure of the reduced cliques $\mathbb{X}\in\mathcal{X}$. 
		\begin{definition}\label{def:label_inversions}
			Consider $(\mathfrak{L}_2(G),\mathcal{X})$ and 
			\begin{equation}
				\mathfrak{V} := \lbrace \Lambda\text{ valid labling of }(\mathfrak{L}_2(G),\mathcal{X})\rbrace.
			\end{equation}
			 For $\mathbb{X}\in\mathcal{X}$, we define the map $\mathcal{F}_{\mathbb{X}}:\mathfrak{V}\to \mathfrak{V}$ for $\Lambda=\left\lbrace  \lbrace \lambda_e\rbrace_{\lbrace e\in \mathbb{Y}\rbrace}\,\big|\,\mathbb{Y}\in\mathcal{X}\right\rbrace$ as  
			 \begin{equation}
			 \mathcal{F}_{\mathbb{X}}(\Lambda):=\left\lbrace  \lbrace \lambda_e\rbrace_{\lbrace e\in \mathbb{Y}\rbrace}\,\big|\,\mathbb{Y}\neq\mathbb{X}\right\rbrace \cup\left\lbrace  \lbrace 1-\lambda_e\rbrace_{\lbrace e\in \mathbb{X}\rbrace}\right\rbrace.
			 \end{equation}			
			We call $\mathcal{F}_{\mathbb{X}}$ a label inversion of the reduced clique $\mathbb{X}\in \mathcal{X}$ over $(\mathfrak{L}_2(G),\mathcal{X})$.
		\end{definition}		
		 Label inversions can be seen in many different ways, from generators of spin dynamics in statistical mechanics to open-closed path flippings in electrical networks and combined with their underlying reduced cycle $\mathbb{X}$ they even can be interpreted as inversions in logic gates with two inputs and two outputs. Label inversions have the following important properties.
		\begin{lemma}\label{lem:intersections_joining_cycles_after_flips}
			Consider $(\mathfrak{L}_2(G),\mathcal{X},\Lambda)$ and let $\mathbb{X}\in\mathcal{X}$. Then, we have the following properties.
			\begin{enumerate}[label=\alph*)]
				\item If $\mathbb{X}$ joins two adjacent cycles in $(\mathfrak{L}_2(G),\mathcal{X},\Lambda)$, then $\mathbb{X}$ is a Type B intersection in $(\mathfrak{L}_2(G),\mathcal{X},\mathcal{F}_{\mathbb{X}}(\Lambda))$.
				\item If $\mathbb{X}$ is a Type B intersection in $(\mathfrak{L}_2(G),\mathcal{X},\Lambda)$, then $\mathbb{X}$ joins two adjacent cycles in $(\mathfrak{L}_2(G),\mathcal{X},\mathcal{F}_{\mathbb{X}}(\Lambda))$.
				\item If $\mathbb{X}$ is a Type A intersection in $(\mathfrak{L}_2(G),\mathcal{X},\Lambda)$, then $\mathbb{X}$ is a Type A intersection $(\mathfrak{L}_2(G),\mathcal{X},\mathcal{F}_{\mathbb{X}}(\Lambda))$.
			\end{enumerate}			  
		\end{lemma}
		These cases can be checked directly from Definition \ref{not:cycle_types} and the associated diagrams, and form the basis for our further arguments on joining and separating cycles. The goal of the whole setup is demonstrated in the following Lemma.
		\begin{lemma}\label{lem:projection_without_intersection_results_in_cycles}
			Let $\Lambda$ a valid labeling such that there are no self-intersections in $(\mathfrak{L}_2(G),\mathcal{X},\Lambda)$. Then, the projection of $\Lambda$ to $G$ under $\pi\circ \chi_{\mathcal{E}}$ is a set of cycles.
		\end{lemma}
		\begin{proof}
			Assume that $(\pi\circ \chi_{\mathcal{E}})(\Lambda)$ implies a closed walk on $G$, which is not a cycle and which, therefore, traverses an edge twice. The inverse operation would imply that there is a vertex in $\mathcal{L}(G)$ such that all incident edges have the same color. This is impossible by assumption on $\Lambda$ since the lift of this vertex would give a self-intersection in $(\mathfrak{L}_2(G),\mathcal{X},\Lambda)$.
		\end{proof}
		Our goal is, therefore, to find a set of labels $\Lambda$, which has no self-intersections. To avoid lengthy wordings in what follows, we use the following definition.
		\begin{definition}
			Given $(\mathfrak{L}_2(G),\mathcal{X},\Lambda)$, let $\mathbb{X}\in\mathcal{X}$ be an intersection of Type A or B in $(\mathfrak{L}_2(G),\mathcal{X},\Lambda)$. We call $\mathbb{X}$ reducible, if there is a finite number of label inversions $\mathcal{F}_{\mathbb{X}_1},\hdots,\mathcal{F}_{\mathbb{X}_n}$ such that $\mathbb{X}$ joins two cycles in $(\mathfrak{L}_2(G),\mathcal{X},\Lambda')$ where $\Lambda':=(\mathcal{F}_{\mathbb{X}_1}\circ\hdots\circ\mathcal{F}_{\mathbb{X}_n})(\Lambda)$.
			In any other case, we call $\mathbb{X}$ irreducible.
		\end{definition}
		Indeed, by Lemma \ref{lem:intersections_joining_cycles_after_flips} all intersections of Type B are reducible, using simply their assigned label inversion $\mathcal{F}_{\mathbb{X}}$. This is not always the case for Type A intersections. The first case, for which reducibility fails, will be discussed next. 
		\begin{lemma}\label{lem:bridge_implies_typeA_label}
			Let $G$ be a simple connected cubic graph. If $G$ has a bridge, then for all valid labelings $\Lambda$ there is a self-intersection $\mathbb{X}$ of Type A in $(\mathfrak{L}_2(G),\mathcal{X},\Lambda)$.
		\end{lemma}
		\begin{proof}
			A bridge in $G$ is a cut-vertex in $\mathcal{L}(G)$ and vice-versa. Call this vertex in $\mathcal{L}(G)$ simply $q$. Then, the vertex $q$ implies by construction a reduced clique in $(\mathfrak{L}_2(G),\mathcal{X},\Lambda)$, which we call $\mathbb{X}_q$. Since any closed trail in $\mathcal{L}(G)$ containing $q$ has to go through $q$ twice, the lift to $\mathbb{X}_q$ implies a Type A self-intersection which cannot be resolved since it separates the cycle into two independent parts which cannot be joined otherwise. 
		\end{proof}
		The obvious consequence is that said Type A intersection is irreducible and it is, therefore, impossible to obtain labelings without self-intersections if the underlying graph $G$ has a bridge. The statement in Lemma \ref{lem:bridge_implies_typeA_label} is, indeed, an equivalence but the second part, stated in Proposition \ref{prop:self_ints_everywhere_imply_bridge}, needs further work which we need to do first.  
		\begin{proposition}\label{prop:self_ints_everywhere_imply_bridge}
			Let $G$ be a simple connected cubic graph. If for all valid labelings $\Lambda$ there is a self-intersection $\mathbb{X}$ of Type A in $(\mathfrak{L}_2(G),\mathcal{X},\Lambda)$, then $G$ has a bridge.
		\end{proposition}
		We have to postpone the proof, which will be done constructively, based on the following results. For parts of it we need to take a more global perspective on the cycles than what we have done so far.
		\begin{definition}\label{def:joining_cycles}
			Let $\gamma_1,\gamma_2\in \Gamma_{\Lambda}$ be non-adjacent cycles. We say that a cycle $\gamma_3$ joins $\gamma_1$ and $\gamma_2$, if $\gamma_3$ is adjacent to both $\gamma_1$ and $\gamma_2$. 
		\end{definition}
		It turns out, that two fixed cycles can always be joined as explained in the following Proposition \ref{prop:joining_cycles_in_cubic_graph_double_line_graph}.
		\begin{proposition}\label{prop:joining_cycles_in_cubic_graph_double_line_graph}
			Consider $(\mathfrak{L}_2(G),\mathcal{X},\Lambda)$ and $\gamma_1,\gamma_2\in \Gamma_{\Lambda}$ be non-adjacent cycles. Then, there is a finite number of label inversions $\mathcal{F}_{\mathbb{X}_1},\hdots,\mathcal{F}_{\mathbb{X}_n}$ and $\Lambda'=(\mathcal{F}_{\mathbb{X}_1}\circ\hdots \circ\mathcal{F}_{\mathbb{X}_{n}})(\Lambda)$ such that $\gamma_1,\gamma_2\in \Gamma_{\Lambda'}$ and $\gamma_1,\gamma_2$ are joined in $(\mathfrak{L}_2(G),\mathcal{X},\Lambda')$ by some cycle $\gamma_3\in\Gamma_{\Lambda'}$. 
		\end{proposition}
		\begin{proof}
			Let $\gamma_1,\gamma_2\in \Gamma_{\Lambda}$ be non-adjacent cycles. From Proposition \ref{prop:reduced_double_line_graph_4_reg_con}, we know that $(\mathfrak{L}_2(G),\mathcal{X})$ is connected and $\gamma_1,\gamma_2$ can be embedded in $(\mathfrak{L}_2(G),\mathcal{X})$. Additionally, consider the vertex induced subgraph $\mathcal{H}$ of $\mathfrak{L}_2(G)$ defined by the vertex set
			\begin{equation}
				\Omega_{\gamma_1,\gamma_2} = \lbrace v\in \mathfrak{L}_2(G)\,|\, v\not\in \gamma_1\cup \gamma_2\rbrace.
			\end{equation}
			Again, by Proposition \ref{prop:reduced_double_line_graph_4_reg_con}, $\mathcal{H}$ is connected. For $\mathbb{X}\in\mathcal{X}$ with $\mathbb{X}\cap \gamma_1 \neq \emptyset$ or $\mathbb{X}\cap \gamma_2 \neq \emptyset$ exactly two vertices $v,w\in \mathbb{X}$ are still present in $\mathcal{H}$ and exactly one edge connecting $v,w$, since $\gamma_1,\gamma_2\in \Gamma_{\Lambda}$ are non-adjacent cycles. Note that all such edges $\langle v,w\rangle$ are open in $(\mathfrak{L}_2(G),\mathcal{X},\Lambda)$. All other $\mathbb{Y}\in\mathcal{X}$ are subgraphs of $\mathcal{H}$. Consequently, we can choose a valid labeling $\Lambda'$ such that $\gamma_1,\gamma_2\in \Gamma_{\Lambda'}$, all $\langle v,w\rangle$ from before are open and all open edges over $\mathcal{H}$ belong to the same cycle $\gamma_3\in \Gamma_{\Lambda'}$. Then, the new cycle $\gamma_3\in \Gamma_{\Lambda'}$ joins $\gamma_1,\gamma_2\in \Gamma_{\Lambda'}$ in $(\mathfrak{L}_2(G),\mathcal{X},\Lambda')$.
		\end{proof}
		Proposition \ref{prop:joining_cycles_in_cubic_graph_double_line_graph} can be interpreted in parallel to the connectedness of $\mathfrak{G}_{\Lambda}$, the graph defined in Lemma \ref{lem:graph_of_cycles} over the set of vertices $\Gamma_{\Lambda}$ for some valid labeling $\Lambda$, where two cycles $\gamma_1,\gamma_2\in\Gamma_{\Lambda}$ form an edge with adjacency given by Definition \ref{def:adjacent_cycles}. 
		It turns out that the all or nothing construction for $\gamma_3$ from the proof of Proposition \ref{prop:joining_cycles_in_cubic_graph_double_line_graph} is not the most appropriate for the proof of Proposition \ref{prop:self_ints_everywhere_imply_bridge}. We discuss a finer approach in Subsection \ref{subsub:monotonicity_type_a_deletion}.
		\begin{figure}[H]
	         \centering
	         \includegraphics[width=0.4\textwidth]{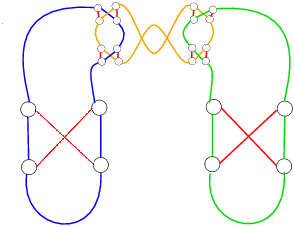}
	         \caption{Cycles $\mathbb{X}_1,\mathbb{X}_2$ containing Type A self-intersections with joining cycle potentially obtained by some label inversions $\mathbb{F}$.}
	         \label{fig:joined_cycles_with possible}
		\end{figure}
		The final preparation step consists in capturing the possible positions of Type A self-intersection. To this end, we employ vertex-cuts in $\mathcal{L}(G)$ which we lift to $(\mathfrak{L}_2(G),\mathcal{X})$ by exploiting the structure of the line graph of a cubic bridgeless graph.
		\begin{definition}\label{def:centered_vertex_cut}
			Let $v\in \mathcal{L}(G)$. Then, we we define a minimal vertex-cut $\mathbb{V}_v$ centered at $v$ as a vertex-cut in $\mathcal{L}(G)$ with $v\in\mathbb{V}_v$, for all $v'\in\mathbb{V}_v$ the set $\mathbb{V}_v\setminus\lbrace v'\rbrace$ is not a vertex-cut and $|\mathbb{V}_v|$ is minimal over all such sets.
		\end{definition}
		We discuss the implications and results of Definition \ref{def:centered_vertex_cut} in Appendix \ref{app:centered_vertex_cut}.
		Evidently, this concept can be lifted to $(\mathfrak{L}_2(G),\mathcal{X})$ by first identifying the edge $e$ in $G$ with a vertex $v_e$ in $\mathcal{L}(G)$ and then associating a reduced clique $\mathbb{X}_v\in\mathcal{X}$ to $v_e$. To this end, we employ the perspective of the line-graph operator discussed in \cite{Har72} based on replacement of each vertex with a complete graph of size $4$ and edge contraction of the old edges. 
		\section{Proof of Proposition \ref{prop:self_ints_everywhere_imply_bridge}}\label{sec:proof_prop}
		We recommend that the reader first goes through Appendix \ref{app:centered_vertex_cut} to familiarize themselves with the concepts around lifted minimal centered vertex-cuts and minimal centered vertex-cuts in the line graph. We use them extensively in the proof. We employ centered vertex-cuts in $\mathcal{L}(G)$ as the central tool and assume that $G$ is bridge-free. The goal is to show that under these conditions any Type A self-intersection can be reduced by label transformations which do not create new Type A self-intersections. This gives a strongly monotonous sequence of the number of Type A self-intersections which by construction converges to $0$.
			\subsection{Reducibility of Type A self-intersections}\label{subsub:reduction_Type_A_self_int}
			To illustrate, we consider first the case of $(\mathfrak{L}_2(G),\mathcal{X},\Lambda)$ such that there is a valid labeling with all open edges being traversed by the same cycle. Pick any two self-intersection $\mathbb{X},\mathbb{Y}\in\mathcal{X}$ as depicted in Figure \ref{fig:dim_one_vertex_cut_lifted_sum_1_1_indep_res}. Then the cycle can be laid out using $\mathbb{X}$ and $\mathbb{Y}$ into an "upper" part and a "lower" part. Inverting the labels of one of self-intersections separates this cycle into two cycles and $\mathbb{X}$ as well as $\mathbb{Y}$ join the two. 
			\begin{figure}[H]
	     \centering
	     \hspace*{35pt}
	     \begin{subfigure}[H]{0.36\textwidth}
	         \centering
	         \includegraphics[width=0.9\textwidth]{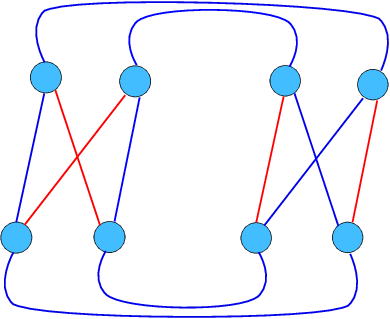}
	         \caption{}
	         \label{fig:dim_one_vertex_cut_lifted_sum_1_1_indep_res}
	     \end{subfigure}
	     \hfill
	     \begin{subfigure}[H]{0.36\textwidth}
	         \centering
	         \includegraphics[width=0.9\textwidth]{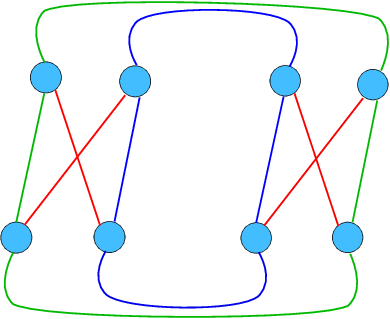}
	         \caption{}
	         \label{fig:dim_one_vertex_cut_lifted_sum_1_1_indep_two_cyc}
	     \end{subfigure}		
	     \hspace*{35pt}
	     \caption{A cycle containing two self-intersections can always be resolved into two cycles. On the left the labels of $\mathbb{X}$ and $\mathbb{Y}$ imply one cycle with self-intersections, which can be split into two joined cycles by label inversion of $\mathbb{X}$ or $\mathbb{Y}$.}
		\label{fig:dim_one_vertex_cut_lifted_full}	
	     \end{figure}
			Heuristically, to prove Proposition \ref{prop:self_ints_everywhere_imply_bridge} we look for a way to join the "upper" and "lower" part of the cycle containing some Type A self-intersection $\mathbb{X}$ and invert the necessary labels. Finally, we have to show that this does not create Type A self-intersections, which leads to a strictly monotonous process converging to $0$. This last part will be done in Subsection \ref{subsub:monotonicity_type_a_deletion}.
			\begin{figure}[H]
			\centering
			\begin{tikzcd}
				 \text{Pick Type A self-intersection }\mathbb{X}_v \arrow[d]\\
				 \text{Lift }\mathbb{V}_v\text{ to }\mathcal{V}_X \arrow[d]\\
				 \text{Choose adapted case }|\mathbb{V}_v| \text{ Figures \ref{fig:dim_two_vertex_cut_implied_cycles} to \ref{fig:dim_four_vertex_cut_lifted_full_implied_cycles}} \arrow[d]\\
				 \text{Resolve as in Figures \ref{fig:dim_two_vertex_cut_lifted_sum_1_1_indep_res} to \ref{fig:dim_four_vertex_cut_lifted_3_1_two_cyc_1}}\\ 
			\end{tikzcd}
			\caption{Proof sketch of Proposition \ref{prop:self_ints_everywhere_imply_bridge} based on reduction of type A self-intersection}
			\label{fig:resolution_description}
		\end{figure}	
		From this point on we can, due to the previous discussion around Figure \ref{fig:dim_one_vertex_cut_lifted_full}, assume that there are at least two cycles induced by the open edges in $(\mathfrak{L}_2(G),\mathcal{X},\Lambda)$. Pick any Type A self-intersection $\mathbb{X}$ in $(\mathfrak{L}_2(G),\mathcal{X},\Lambda)$ and associate it to a vertex in $v\in \mathcal{L}(G)$ as discussed in \cite{Har72}. We put the emphasis on this link by changing the name of $\mathbb{X}$ to $\mathbb{X}_v$. \par 
		Denote by $\mathbb{V}_v$ a minimal vertex cut in $\mathcal{L}(G)$ centered at $v$ and by $\mathcal{V}\subset\mathcal{X}$ the associated reduced cliques in $(\mathfrak{L}_2(G),\mathcal{X},\Lambda)$. Note that necessarily by construction $\mathbb{X}_v\in\mathcal{V}$. From Appendix \ref{app:centered_vertex_cut} and Proposition \ref{prop:joining_cycles_in_cubic_graph_double_line_graph} we can conclude that in any of the two connected components $C_1,C_2$, which the cycles are going through, the lifted vertex cut are either adjacent or are joined by a cycle. Based on the following catalog of lifted vertex cut structures we go through all possible cases. They are ordered by the size of the vertex cut $\mathbb{V}_v$. We start with the case $|\mathbb{V}_v|=2$.
			\begin{figure}[H]
	         \centering
	         \includegraphics[width=0.36\textwidth]{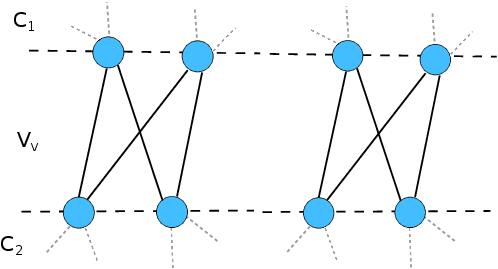}
	     \caption{All possible situations for vertex cuts $\mathbb{V}_v$ with $|\mathbb{V}_v|=2$. See Appendix \ref{app:centered_vertex_cut} for the case of both cut vertices being neighbors in $\mathcal{L}(G)$.}
		\label{fig:dim_two_vertex_cut_lifted_sum_1_1}
	\end{figure}
			We continue with the case $|\mathbb{V}_v|=3$.
			\begin{figure}[H]
	     \centering
	     \hspace*{25pt}
	     \begin{subfigure}[H]{0.45\textwidth}
	         \centering
	         \includegraphics[width=\textwidth]{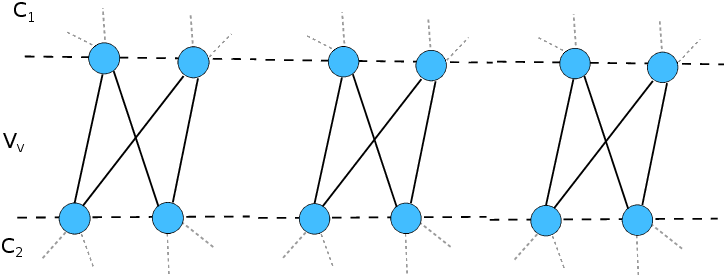}
	         \caption{}
	         \label{fig:dim_three_vertex_cut_lifted_sum_1_1_1}
	     \end{subfigure}
	     \hfill
	     \begin{subfigure}[H]{0.36\textwidth}
	         \centering
	         \includegraphics[width=\textwidth]{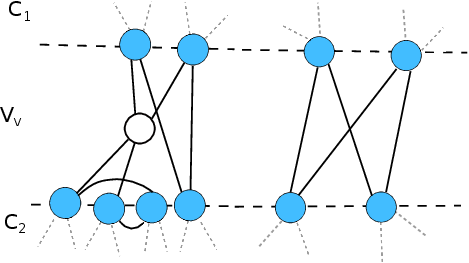}
	         \caption{}
	         \label{fig:dim_three_vertex_cut_lifted_sum_2_1}
	     \end{subfigure}
	     \hspace*{25pt}
	     \vskip\baselineskip
	     \begin{subfigure}[H]{0.24\textwidth}
	         \centering
	         \includegraphics[width=\textwidth]{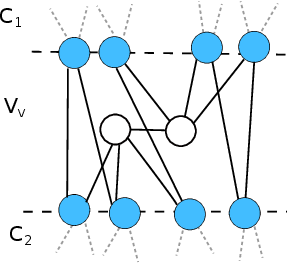}
	         \caption{}
	         \label{fig:dim_three_vertex_cut_lifted}
	     \end{subfigure}
	     \caption{All possible situations for vertex cuts $\mathbb{V}_v$ with $|\mathbb{V}_v|=3$.}
		\label{fig:dim_three_vertex_cut_lifted_full}
	\end{figure}
			Finally, the case $|\mathbb{V}_v|=4$.
			 \begin{figure}[H]
	     \centering
	     \begin{subfigure}[H]{0.45\textwidth}
	         \centering
	         \includegraphics[width=\textwidth]{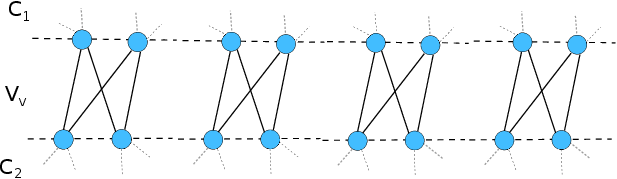}
	         \caption{}
	         \label{fig:dim_four_vertex_cut_lifted_sum_1_1_1_1}
	     \end{subfigure}
	     \hfill
	     \begin{subfigure}[H]{0.45\textwidth}
	         \centering
	         \includegraphics[width=\textwidth]{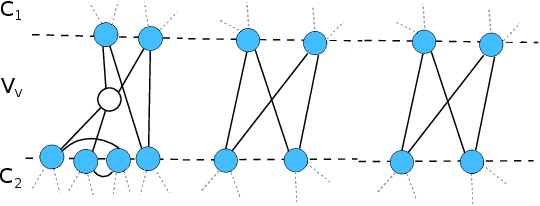}
	         \caption{}
	         \label{fig:dim_four_vertex_cut_lifted_sum_2_1_1}
	     \end{subfigure}
	     \vskip\baselineskip
	     \hspace*{35pt}
	     \begin{subfigure}[H]{0.36\textwidth}
	         \centering
	         \includegraphics[width=\textwidth]{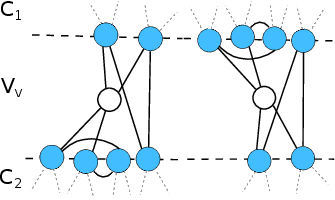}
	         \caption{}
	         \label{fig:dim_four_vertex_cut_lifted_sum_2_2}
	     \end{subfigure}
	     \hfill
	     \begin{subfigure}[H]{0.36\textwidth}
	         \centering
	         \includegraphics[width=\textwidth]{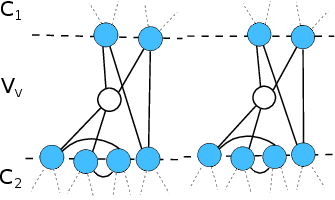}
	         \caption{}
	         \label{fig:dim_four_vertex_cut_lifted_sum_2_2_aligned}
	     \end{subfigure}
	     \hspace*{35pt}
	     \vskip\baselineskip
	     \hspace*{25pt}
	     \begin{subfigure}[H]{0.45\textwidth}
	         \centering
	         \includegraphics[width=\textwidth]{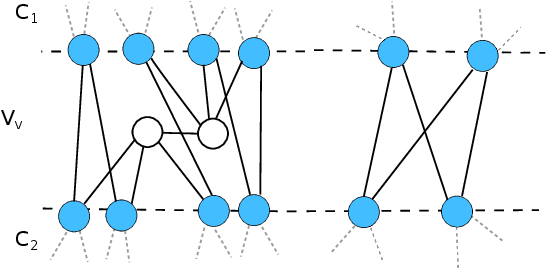}
	         \caption{}
	         \label{fig:dim_four_vertex_cut_lifted_sum_3_1}
	     \end{subfigure}
	     \hfill
	     \begin{subfigure}[H]{0.36\textwidth}
	         \centering
	         \fcolorbox{red}{white}{\includegraphics[width=\textwidth]{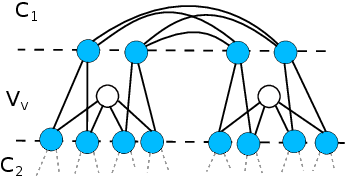}}
	         \caption{}
	         \label{fig:dim_four_vertex_cut_lifted}
	     \end{subfigure}
	     \hspace*{25pt}
	     \caption{All possible situations for vertex cuts $\mathbb{V}_v$ with $|\mathbb{V}_v|=4$.}
		\label{fig:dim_four_vertex_cut_lifted_full}
	\end{figure}
	In the case $|\mathbb{V}_v|=4$, we may always choose the case represented in Figure \ref{fig:dim_four_vertex_cut_lifted}, as a lift of a minimal vertex cut, to reduce type A self-intersections. This is due to the fact, that $G$ is assumed to be $3$ regular and, hence, its line graph is $4$ regular. Any neighborhood of any vertex $v$ in $\mathcal{L}(G)$ induces, hence, a minimal vertex cut containing $v$. This reduces in what follows the number of cases tremendously, see Figure \ref{fig:dim_four_vertex_cut_lifted_full_implied_cycles}, which only needs to focus on cycles induced by Figure \ref{fig:dim_four_vertex_cut_lifted} and not on all possible cases represented in Figure \ref{fig:dim_four_vertex_cut_lifted_full}.\par
	\subsection{Generic view of cycles associated to lifted vertex cuts} 
	In this subsection, we focus on the set of cycles $\Gamma_{\Lambda}$ in $(\mathfrak{L}_2(G),\mathcal{X},\Lambda)$, all types given Definition \ref{not:cycle_types}, and as proven in Lemma \ref{lem:open_edges_disjoint_union_of_cycles}. We can without knowledge of the local structure of the cycles, based on the lifted centered vertex cut $\mathcal{V}$, consider all cases of possible global cycle behavior. We call this the "generic" view of the cycles, which focuses only on some local structures within a lifted vertex cut. The following generic cycles can be associated to the corresponding vertex cuts. We can assume two connected components $C_1,C_2$, referring to the discussion at the end of Appendix \ref{app:centered_vertex_cut}, which is mostly due to the absence of bridges. This is a combinatorial exercise where symmetric situations with respect to the structure of the cycles are identified.
	\begin{figure}[H]
	     \centering
	     \hspace*{35pt}
	     \begin{subfigure}[H]{0.36\textwidth}
	         \centering
	         \includegraphics[width=\textwidth]{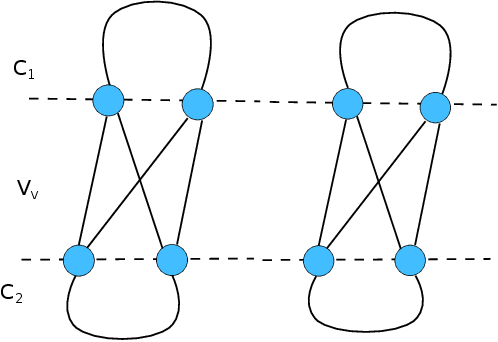}
	         \caption{}
	         \label{fig:dim_two_vertex_cut_lifted_sum_1_1_indep}
	     \end{subfigure}
	     \hfill
	     \begin{subfigure}[H]{0.36\textwidth}
	         \centering
	         \includegraphics[width=\textwidth]{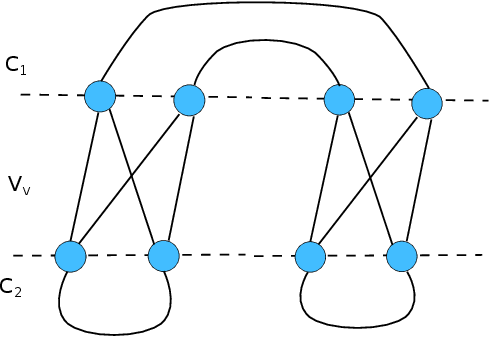}
	         \caption{}
	         \label{fig:dim_two_vertex_cut_lifted_sum_1_1_joined}
	     \end{subfigure}	
	     \hspace*{35pt}     
	     \caption{All possible generic cycles for vertex cuts $\mathbb{V}_v$ with $|\mathbb{V}_v|=2$ based on Figure \ref{fig:dim_two_vertex_cut_lifted_sum_1_1}.}
	     \label{fig:dim_two_vertex_cut_implied_cycles}
	     \end{figure}
	The following sub-cases arise under generic cycle representations. We go from Figure \ref{fig:dim_two_vertex_cut_lifted_sum_1_1} to Figure \ref{fig:dim_four_vertex_cut_lifted} through all relevant cases and their implied generic cycles.
	\begin{figure}[H]
	     \centering
	     \begin{subfigure}[H]{0.45\textwidth}
	         \centering
	         \includegraphics[width=\textwidth]{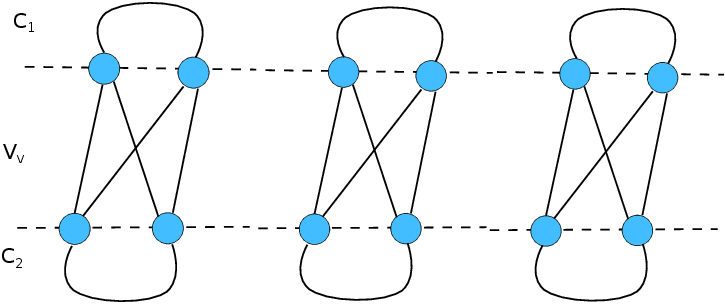}
	         \caption{}
	         \label{fig:dim_three_vertex_cut_lifted_sum_1_1_1_indep}
	     \end{subfigure}
	     \hfill
	     \begin{subfigure}[H]{0.45\textwidth}
	         \centering
	         \includegraphics[width=\textwidth]{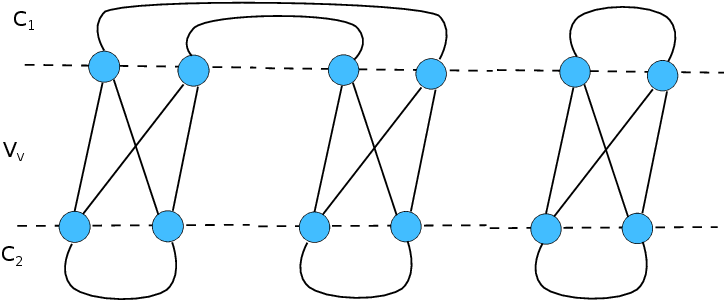}
	         \caption{}
	         \label{fig:dim_three_vertex_cut_lifted_sum_1_1_1_join_1_2}
	     \end{subfigure}
	     \vskip\baselineskip
	     \begin{subfigure}[H]{0.45\textwidth}
	         \centering
	         \includegraphics[width=\textwidth]{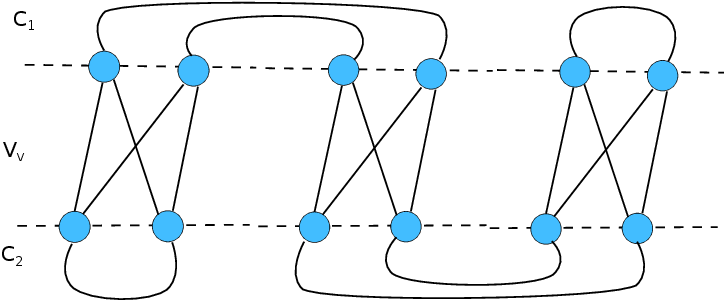}
	         \caption{}
	         \label{fig:dim_three_vertex_cut_lifted_sum_1_1_1_join_1_2_a_5_6}
	     \end{subfigure}
	     \hfill
	     \begin{subfigure}[H]{0.45\textwidth}
	         \centering
	         \includegraphics[width=\textwidth]{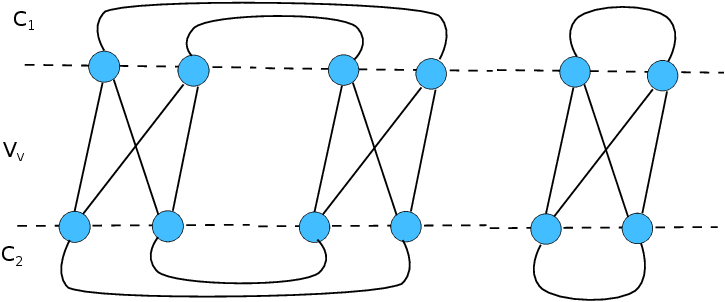}
	         \caption{}
	         \label{fig:dim_three_vertex_cut_lifted_sum_1_1_1_double_indep}
	     \end{subfigure}
	     \vskip\baselineskip
	     \begin{subfigure}[H]{0.45\textwidth}
	         \centering
	         \includegraphics[width=\textwidth]{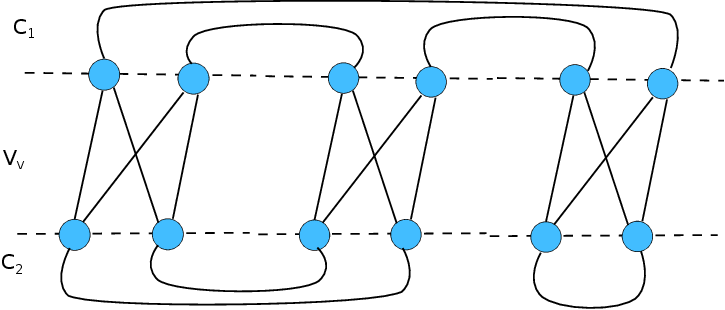}
	         \caption{}
	         \label{fig:dim_three_vertex_cut_lifted_sum_1_1_1_one_sided_double}
	     \end{subfigure}
	     \hfill
	     \begin{subfigure}[H]{0.45\textwidth}
	         \centering
	         \includegraphics[width=\textwidth]{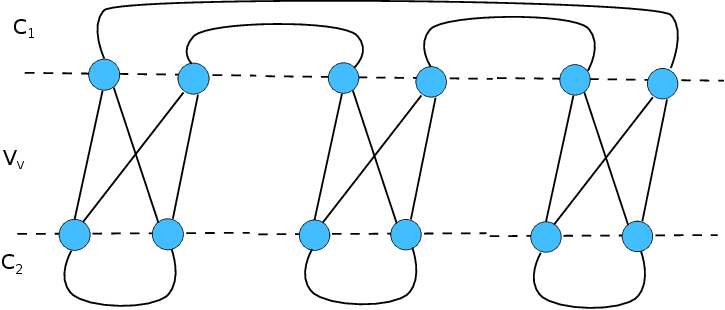}
	         \caption{}
	         \label{fig:dim_three_vertex_cut_lifted_sum_1_1_1_one_sided_indep}
	     \end{subfigure}
	     \caption{All possible generic cycles for vertex cuts $\mathbb{V}_v$ with $|\mathbb{V}_v|=3$ based on Figure \ref{fig:dim_three_vertex_cut_lifted_sum_1_1_1}.}
		\label{fig:dim_three_vertex_cut_lifted_sum_1_1_1_implied_cycles}
	\end{figure}
	\begin{figure}[H]
	     \centering
	     \hspace*{35pt}
	     \begin{subfigure}[H]{0.36\textwidth}
	         \centering
	         \includegraphics[width=\textwidth]{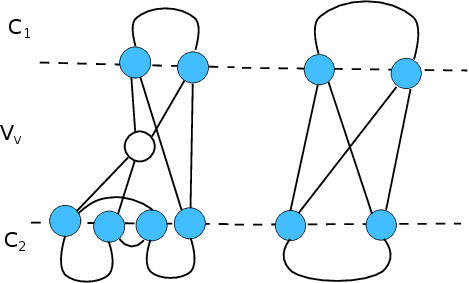}
	         \caption{}
	         \label{fig:dim_three_vertex_cut_lifted_sum_2_1_indep}
	     \end{subfigure}
	     \hfill
	     \begin{subfigure}[H]{0.36\textwidth}
	         \centering
	         \includegraphics[width=\textwidth]{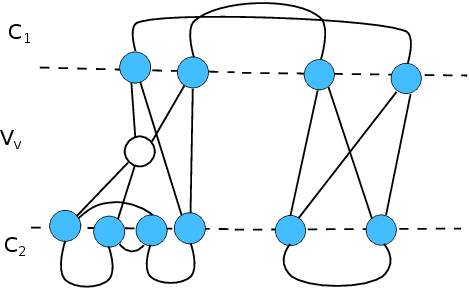}
	         \caption{}
	         \label{fig:dim_three_vertex_cut_lifted_sum_2_1_small_connect_indep}
	     \end{subfigure}
	     \hspace*{35pt}
	     \vskip\baselineskip
	     \hspace*{35pt}
	     \begin{subfigure}[H]{0.36\textwidth}
	         \centering
	         \includegraphics[width=\textwidth]{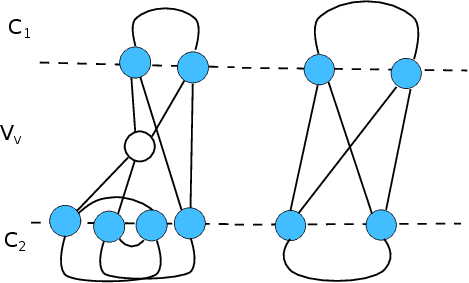}
	         \caption{}
	         \label{fig:dim_three_vertex_cut_lifted_sum_2_1_comp_connect_indep}
	     \end{subfigure}
	     \hfill
	     \begin{subfigure}[H]{0.36\textwidth}
	         \centering
	         \includegraphics[width=\textwidth]{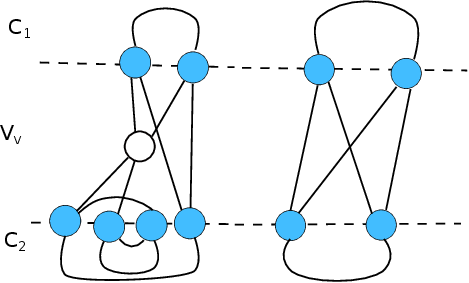}
	         \caption{}
	         \label{fig:dim_three_vertex_cut_lifted_sum_2_1_comp_connect_indep_2}
	     \end{subfigure}
	     \hspace*{35pt}
	     \vskip\baselineskip
	     \hspace*{35pt}
	     \begin{subfigure}[H]{0.36\textwidth}
	         \centering
	         \includegraphics[width=\textwidth]{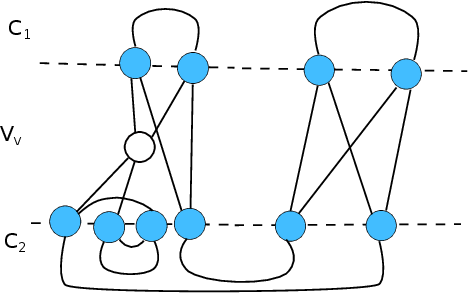}
	         \caption{}
	         \label{fig:dim_three_vertex_cut_lifted_sum_2_1_large_connect_indep_1}
	     \end{subfigure}
	     \hfill
	     \begin{subfigure}[H]{0.36\textwidth}
	         \centering
	         \includegraphics[width=\textwidth]{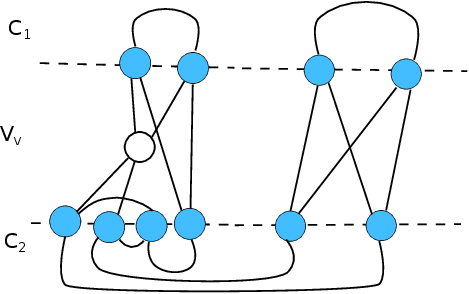}
	         \caption{}
	         \label{fig:dim_three_vertex_cut_lifted_sum_2_1_large_connect_indep_2}
	     \end{subfigure}
	     \hspace*{35pt}
	     \caption{All possible generic cycles for vertex cuts $\mathbb{V}_v$ with $|\mathbb{V}_v|=3$ based on Figure \ref{fig:dim_three_vertex_cut_lifted_sum_2_1}.}
		\label{fig:dim_three_vertex_cut_lifted_sum_2_1_implied_cycles}
	\end{figure}
	\begin{figure}[H]
	     \centering
	     \begin{subfigure}[H]{0.28\textwidth}
	         \centering
	         \includegraphics[width=\textwidth]{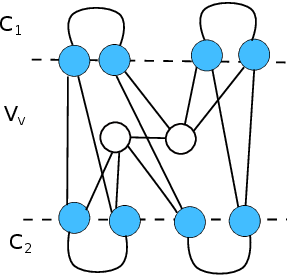}
	         \caption{}
	         \label{fig:dim_three_vertex_cut_lifted_all_loop}
	     \end{subfigure}
	     \hfill
	     \begin{subfigure}[H]{0.28\textwidth}
	         \centering
	         \includegraphics[width=\textwidth]{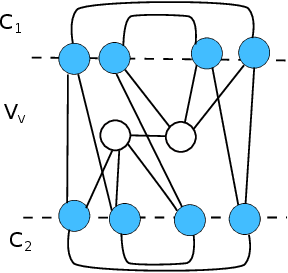}
	         \caption{}
	         \label{fig:dim_three_vertex_cut_lifted_cross_long_cross_long}
	     \end{subfigure}
	     \hfill
	     \begin{subfigure}[H]{0.28\textwidth}
	         \centering
	         \includegraphics[width=\textwidth]{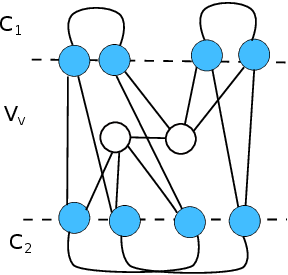}
	         \caption{}
	         \label{fig:dim_three_vertex_cut_lifted_all_loop_cross}
	     \end{subfigure}
	     \vskip\baselineskip
	     \begin{subfigure}[H]{0.28\textwidth}
	         \centering
	         \includegraphics[width=\textwidth]{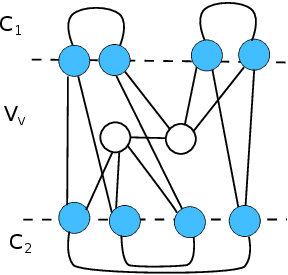}
	         \caption{}
	         \label{fig:dim_three_vertex_cut_lifted_all_loop_cross_long}
	     \end{subfigure}
	     \hfill
	     \begin{subfigure}[H]{0.28\textwidth}
	         \centering
	         \includegraphics[width=\textwidth]{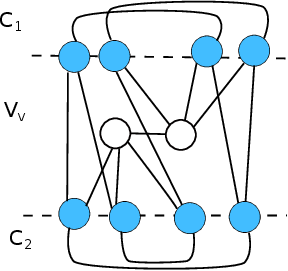}
	         \caption{}
	         \label{fig:dim_three_vertex_cut_lifted_cross_cross_long}
	     \end{subfigure}
	     \hfill
	     \begin{subfigure}[H]{0.28\textwidth}
	         \centering
	         \includegraphics[width=\textwidth]{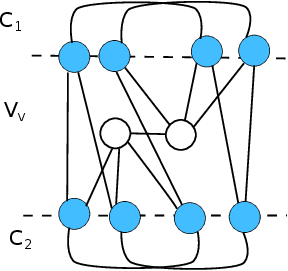}
	         \caption{}
	         \label{fig:dim_three_vertex_cut_lifted_cross_cross}
	     \end{subfigure}
	     \caption{All possible generic cycles for vertex cuts $\mathbb{V}_v$ with $|\mathbb{V}_v|=3$ based on Figure \ref{fig:dim_three_vertex_cut_lifted}.}
		\label{fig:dim_three_vertex_cut_lifted_full_implied_cycles}
	\end{figure}
	\begin{figure}[H]
	     \centering
	     \hspace*{35pt}
	     \begin{subfigure}[H]{0.36\textwidth}
	         \centering
	         \includegraphics[width=\textwidth]{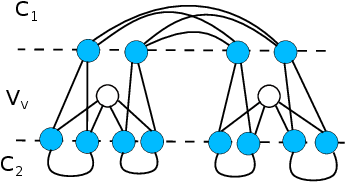}
	         \caption{}
	         \label{fig:dim_four_vertex_cut_lifted_indep}
	     \end{subfigure}
	     \hfill
	     \begin{subfigure}[H]{0.36\textwidth}
	         \centering
	         \includegraphics[width=\textwidth]{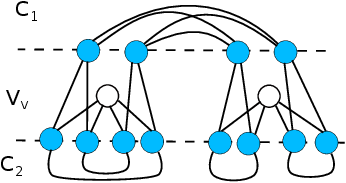}
	         \caption{}
	         \label{fig:dim_four_vertex_cut_lifted_indep_2}
	     \end{subfigure}
	     \hspace*{35pt}
	     \vskip\baselineskip
	     \hspace*{35pt}
	     \begin{subfigure}[H]{0.36\textwidth}
	         \centering
	         \includegraphics[width=\textwidth]{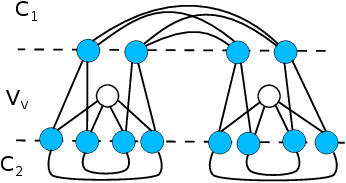}
	         \caption{}
	         \label{fig:dim_four_vertex_cut_lifted_2_2}
	     \end{subfigure}
	     \hfill
	     \begin{subfigure}[H]{0.36\textwidth}
	         \centering
	         \includegraphics[width=\textwidth]{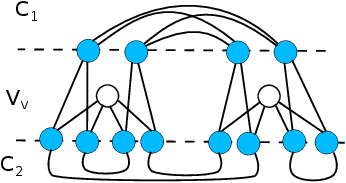}
	         \caption{}
	         \label{fig:dim_four_vertex_cut_lifted_3_1}
	     \end{subfigure}
	     \hspace*{35pt}
	     \caption{All possible generic cycles for vertex cuts $\mathbb{V}_v$ with $|\mathbb{V}_v|=4$ based on Figure \ref{fig:dim_four_vertex_cut_lifted}.}
		\label{fig:dim_four_vertex_cut_lifted_full_implied_cycles}
	\end{figure}
		\subsection{Type A resolution through $\mathfrak{G}_{\Lambda}$} 
		We approach now the reduction of the Type A self-intersections. We employ the intuition developed at the beginning of the proof to connect the "upper" part and "lower" part of the cycle containing the Type A self-intersection which we want to reduce. This connection is always possible due to the properties of $\mathfrak{G}_{\Lambda}$ discussed in Lemma \ref{lem:graph_of_cycles}. We go through all cases presented in Figures \ref{fig:dim_two_vertex_cut_implied_cycles} to \ref{fig:dim_four_vertex_cut_lifted_full_implied_cycles}. In the left column we present the case containing Type A self-intersections and in the right column the resolved variant, using a connecting cycle, which we obtain from Lemma \ref{lem:graph_of_cycles}. In the cases presented in Figures \ref{fig:dim_three_vertex_cut_lifted_cross_long_cross_long_two_cyc}, \ref{fig:dim_three_vertex_cut_lifted_cross_cross_long_two_cyc} and \ref{fig:dim_three_vertex_cut_lifted_cross_cross_two_cyc}, where all self-intersections in the lifted centered vertex cut are Type B, we show only a resolved version. 
			\begin{figure}[H]
	     \centering
	     \hspace*{35pt}
	     \begin{subfigure}[H]{0.36\textwidth}
	         \centering
	         \includegraphics[width=\textwidth]{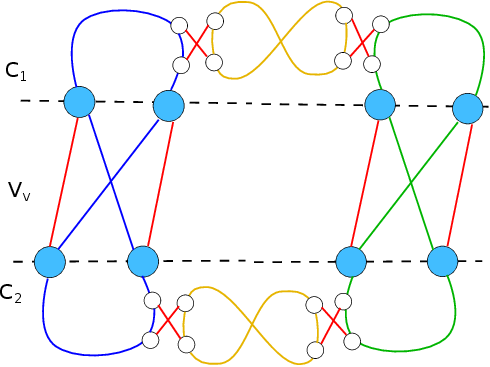}
	         \caption{}
	         \label{fig:dim_two_vertex_cut_lifted_sum_1_1_indep_res}
	     \end{subfigure}
	     \hfill
	     \begin{subfigure}[H]{0.36\textwidth}
	         \centering
	         \includegraphics[width=\textwidth]{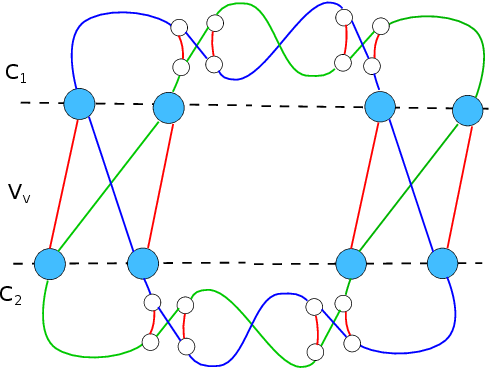}
	         \caption{}
	         \label{fig:dim_two_vertex_cut_lifted_sum_1_1_indep_two_cyc}
	     \end{subfigure}
	     \hspace*{35pt}
	     \vskip\baselineskip
	     \hspace*{35pt}
	     \begin{subfigure}[H]{0.36\textwidth}
	         \centering
	         \includegraphics[width=\textwidth]{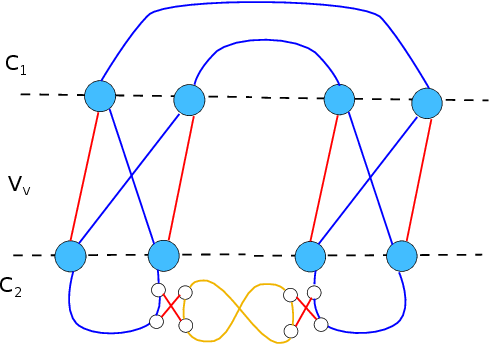}
	         \caption{}
	         \label{fig:dim_two_vertex_cut_lifted_sum_1_1_joined_res}
	     \end{subfigure}
	     \hfill
	     \begin{subfigure}[H]{0.36\textwidth}
	         \centering
	         \includegraphics[width=\textwidth]{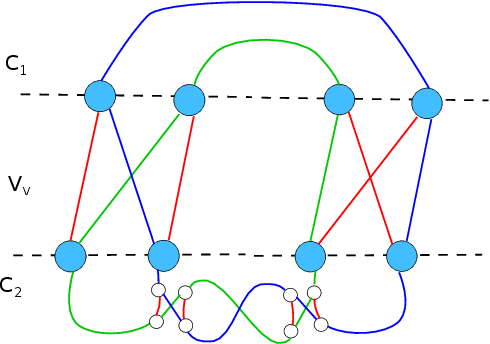}
	         \caption{}
	         \label{fig:dim_two_vertex_cut_lifted_sum_1_1_joined_two_cyc}
	     \end{subfigure}
	     \hspace*{35pt}
	     \caption{Resolutions of Figure \ref{fig:dim_two_vertex_cut_implied_cycles}.}
	     \label{fig:dim_two_vertex_cut_lifted_sum_1_1_resolved}
	     \end{figure}
	     Note that the all cases in Figure \ref{fig:dim_three_vertex_cut_lifted_sum_1_1_1_implied_cycles} can be resolved as is done in Figure \ref{fig:dim_two_vertex_cut_lifted_sum_1_1_resolved} no matter which reduced clique in the lift of $\mathbb{V}_v$ is considered to be the lift of $v$. Moreover, in all cases the existence of the necessary cycles in the lifts of $C_1$ and $C_2$, is assured by Proposition \ref{prop:joining_cycles_in_cubic_graph_double_line_graph} and the fact that traversals between the lifts of $C_1$ and $C_2$ are only possible through the lift of $\mathbb{V}_v$. But the traversals are fixed in each figure, which implies the existence of the necessary structures in $C_1$ and $C_2$, respectively.\par 
	     We continue with the discussion of the cases shown in Figure \ref{fig:dim_three_vertex_cut_lifted_sum_2_1_implied_cycles}. We want to emphasize, that the goal is to reduce Type A self-intersections within the lifted vertex cut $\mathbb{V}_v$. If there are still present in the designs, they can be reduced by another iteration of the constructions within the respective case of $|\mathbb{V}_v|$. For example, the resolution of Figure \ref{fig:dim_four_vertex_cut_lifted_indep_2} is a combination of Figures \ref{fig:dim_four_vertex_cut_lifted_indep_res} and \ref{fig:dim_four_vertex_cut_lifted_2_2_two_cyc}. We only need to make sure, subsequently, that no new Type A self-intersection are created. This constitutes the last part of the proof in Subsection \ref{subsub:monotonicity_type_a_deletion}.
	     \begin{figure}[H]
	     \centering
	     \hspace*{35pt}
	     \begin{subfigure}[H]{0.36\textwidth}
	         \centering
	         \includegraphics[width=\textwidth]{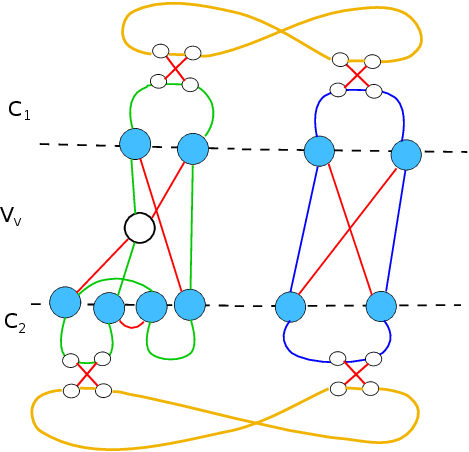}
	         \caption{}
	         \label{fig:dim_three_vertex_cut_lifted_sum_2_1_indep_res_2}
	     \end{subfigure}
	     \hfill
	     \begin{subfigure}[H]{0.36\textwidth}
	         \centering
	         \includegraphics[width=\textwidth]{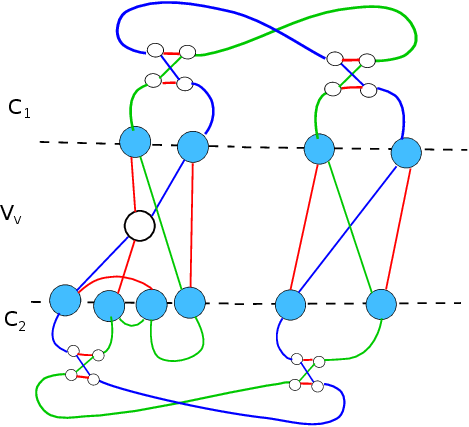}
	         \caption{}
	         \label{fig:dim_three_vertex_cut_lifted_sum_2_1_indep_two_cyc_2}
	     \end{subfigure}
	     \hspace*{35pt}
	     \vskip\baselineskip
	     \hspace*{35pt}
	     \begin{subfigure}[H]{0.36\textwidth}
	         \centering
	         \includegraphics[width=\textwidth]{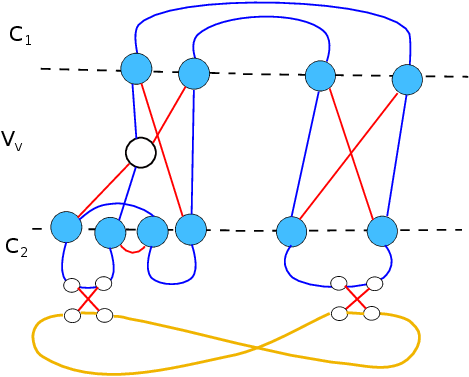}
	         \caption{}
	         \label{fig:dim_three_vertex_cut_lifted_sum_2_1_small_connect_indep_res}
	     \end{subfigure}
	     \hfill
	     \begin{subfigure}[H]{0.36\textwidth}
	         \centering
	         \includegraphics[width=\textwidth]{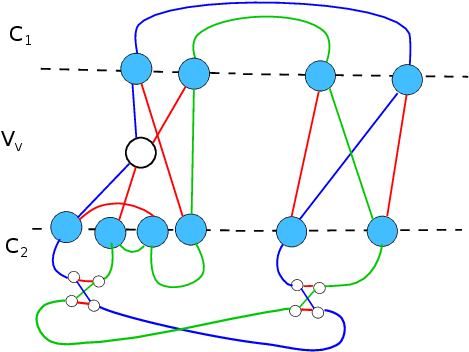}
	         \caption{}
	         \label{fig:dim_three_vertex_cut_lifted_sum_2_1_small_connect_indep_two_cyc}
	     \end{subfigure}
	     \hspace*{35pt}
	     \vskip\baselineskip
	     \hspace*{35pt}
	     \begin{subfigure}[H]{0.36\textwidth}
	         \centering
	         \includegraphics[width=\textwidth]{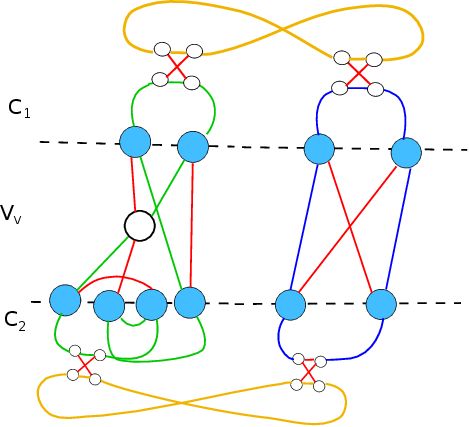}
	         \caption{}
	         \label{fig:dim_three_vertex_cut_lifted_sum_2_1_comp_connect_indep_res}
	     \end{subfigure}
	     \hfill
	     \begin{subfigure}[H]{0.36\textwidth}
	         \centering
	         \includegraphics[width=\textwidth]{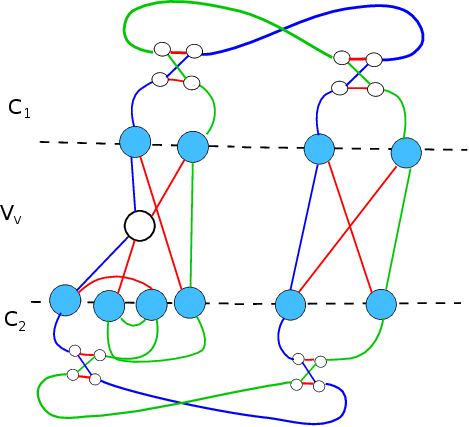}
	         \caption{}
	         \label{fig:dim_three_vertex_cut_lifted_sum_2_1_comp_connect_indep_two_cyc}
	     \end{subfigure}
	     \hspace*{35pt}
	     \caption{Resolutions of Figures \ref{fig:dim_three_vertex_cut_lifted_sum_2_1_indep} to \ref{fig:dim_three_vertex_cut_lifted_sum_2_1_comp_connect_indep}. }
	     \label{fig:dim_three_vertex_cut_lifted_full_implied_cycles_1_resolved}
	     \end{figure}
	     \begin{figure}[H]
	     \centering
	     \hspace*{35pt}
	     \begin{subfigure}[H]{0.36\textwidth}
	         \centering
	         \includegraphics[width=\textwidth]{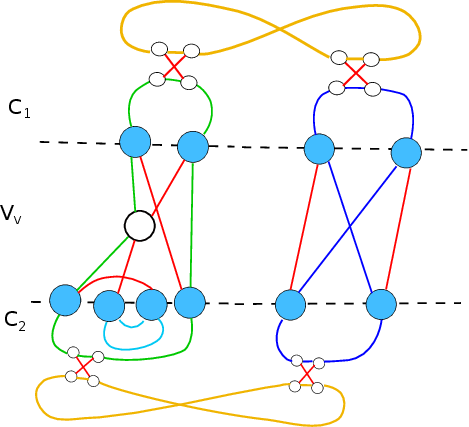}
	         \caption{}
	         \label{fig:dim_three_vertex_cut_lifted_sum_2_1_comp_connect_indep_2_res}
	     \end{subfigure}
	     \hfill
	     \begin{subfigure}[H]{0.36\textwidth}
	         \centering
	         \includegraphics[width=\textwidth]{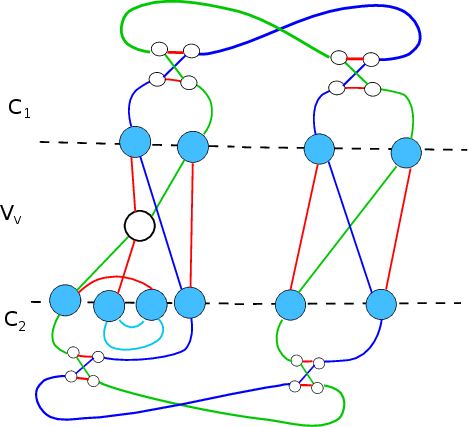}
	         \caption{}
	         \label{fig:dim_three_vertex_cut_lifted_sum_2_1_comp_connect_indep_2_two_cyc}
	     \end{subfigure}
	     \hspace*{35pt}
	     \vskip\baselineskip
	     \hspace*{35pt}
	     \begin{subfigure}[H]{0.36\textwidth}
	         \centering
	         \includegraphics[width=\textwidth]{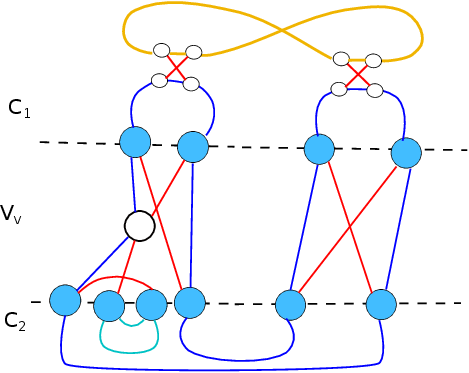}
	         \caption{}
	         \label{fig:dim_three_vertex_cut_lifted_sum_2_1_large_connect_indep_1_res}
	     \end{subfigure}
	     \hfill
	     \begin{subfigure}[H]{0.36\textwidth}
	         \centering
	         \includegraphics[width=\textwidth]{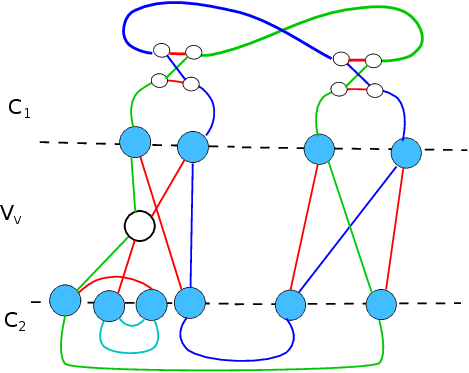}
	         \caption{}
	         \label{fig:dim_three_vertex_cut_lifted_sum_2_1_large_connect_indep_1_two_cyc}
	     \end{subfigure}
	     \hspace*{35pt}
	     \vskip\baselineskip
	     \hspace*{35pt}
	     \begin{subfigure}[H]{0.36\textwidth}
	         \centering
	         \includegraphics[width=\textwidth]{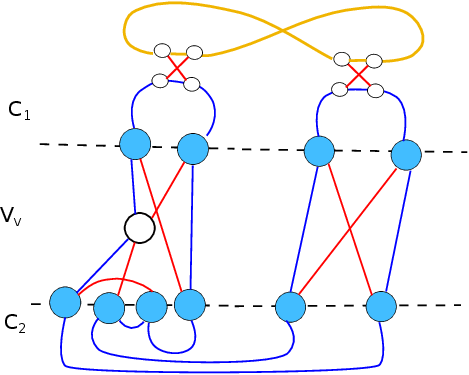}
	         \caption{}
	         \label{fig:dim_three_vertex_cut_lifted_sum_2_1_large_connect_indep_2_res}
	     \end{subfigure}
	     \hfill
	     \begin{subfigure}[H]{0.36\textwidth}
	         \centering
	         \includegraphics[width=\textwidth]{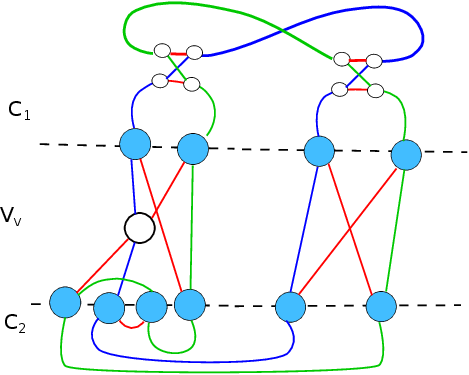}
	         \caption{}
	         \label{fig:dim_three_vertex_cut_lifted_sum_2_1_large_connect_indep_2_two_cyc}
	     \end{subfigure}
	     \hspace*{35pt}
	     \caption{Resolutions of Figures \ref{fig:dim_three_vertex_cut_lifted_sum_2_1_comp_connect_indep_2} to \ref{fig:dim_three_vertex_cut_lifted_sum_2_1_large_connect_indep_2}. }
	     \label{fig:dim_three_vertex_cut_lifted_full_implied_cycles_1_resolved_1}
	     \end{figure}
	     The next step will focus on the cases shown in Figure \ref{fig:dim_three_vertex_cut_lifted_full_implied_cycles}.
	     \begin{figure}[H]
	     \centering
	     \hspace*{35pt}
	     \begin{subfigure}[H]{0.36\textwidth}
	         \centering
	         \includegraphics[width=\textwidth]{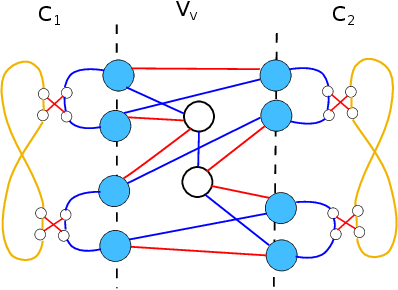}
	         \caption{}
	         \label{fig:dim_three_vertex_cut_lifted_all_loop_res_1}
	     \end{subfigure}
	     \hfill
	     \begin{subfigure}[H]{0.36\textwidth}
	         \centering
	         \includegraphics[width=\textwidth]{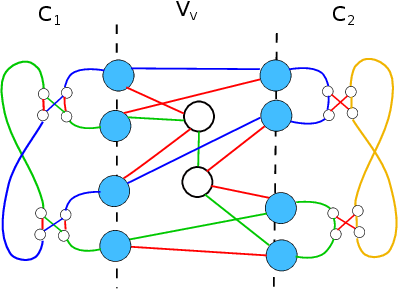}
	         \caption{}
	         \label{fig:dim_three_vertex_cut_lifted_all_loop_two_cyc_1}
	     \end{subfigure}
	     \hspace*{35pt}
	     \vskip\baselineskip
	     \begin{subfigure}[H]{0.28\textwidth}
	         \centering
	         \includegraphics[width=\textwidth]{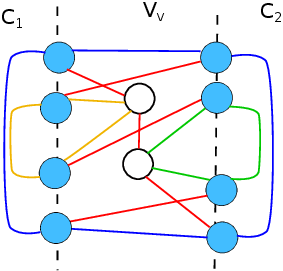}
	         \caption{}
	         \label{fig:dim_three_vertex_cut_lifted_cross_long_cross_long_two_cyc}
	     \end{subfigure}
	     \vskip\baselineskip
	     \hspace*{35pt}
	     \begin{subfigure}[H]{0.36\textwidth}
	         \centering
	         \includegraphics[width=\textwidth]{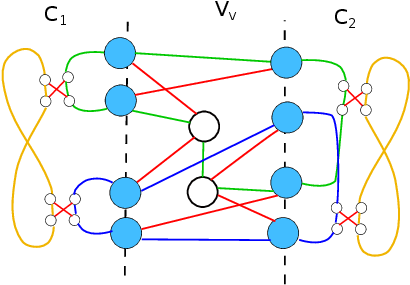}
	         \caption{}
	         \label{fig:dim_three_vertex_cut_lifted_all_loop_cross_res}
	     \end{subfigure}
	     \hfill
	     \begin{subfigure}[H]{0.36\textwidth}
	         \centering
	         \includegraphics[width=\textwidth]{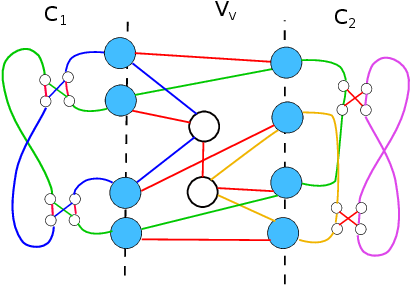}
	         \caption{}
	         \label{fig:dim_three_vertex_cut_lifted_all_loop_cross_two_cyc}
	     \end{subfigure}
	     \hspace*{35pt}
	     \caption{Resolutions of Figures \ref{fig:dim_three_vertex_cut_lifted_all_loop} to \ref{fig:dim_three_vertex_cut_lifted_all_loop_cross}. Some are turned by 90° for aesthetic reasons. Note that in Figure \ref{fig:dim_three_vertex_cut_lifted_cross_long_cross_long_two_cyc} there is a labeling within the lifted vertex cut such that any self-intersection separate into independent cycles.}
	     \label{fig:dim_three_vertex_cut_lifted_full_implied_cycles_2_resolved}
	     \end{figure}
	     \begin{figure}[H]
	     \centering
	     \hspace*{35pt}
	     \begin{subfigure}[H]{0.36\textwidth}
	         \centering
	         \includegraphics[width=\textwidth]{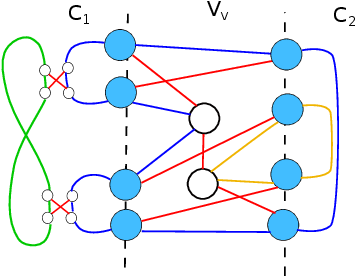}
	         \caption{}
	         \label{fig:dim_three_vertex_cut_lifted_all_loop_cross_long_res}
	     \end{subfigure}
	     \hfill
	     \begin{subfigure}[H]{0.36\textwidth}
	         \centering
	         \includegraphics[width=\textwidth]{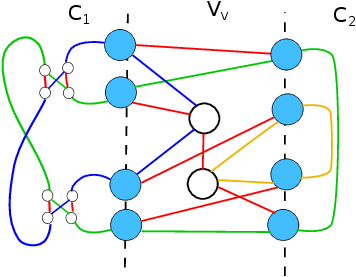}
	         \caption{}
	         \label{fig:dim_three_vertex_cut_lifted_all_loop_cross_long_two_cyc}
	     \end{subfigure}
	     \hspace*{35pt}
	     \vskip\baselineskip
	     \hspace*{35pt}
	     \begin{subfigure}[H]{0.28\textwidth}
	         \centering
	         \includegraphics[width=\textwidth]{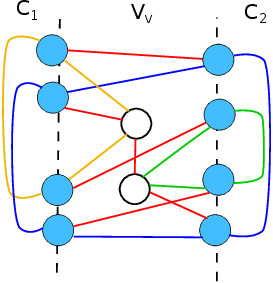}
	         \caption{}
	         \label{fig:dim_three_vertex_cut_lifted_cross_cross_long_two_cyc}
	     \end{subfigure}
	     \hspace*{35pt}
	     \vskip\baselineskip
	     \hspace*{35pt}
	     \begin{subfigure}[H]{0.28\textwidth}
	         \centering
	         \includegraphics[width=\textwidth]{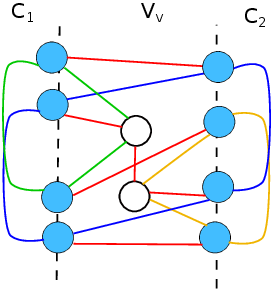}
	         \caption{}
	         \label{fig:dim_three_vertex_cut_lifted_cross_cross_two_cyc}
	     \end{subfigure}
	     \hspace*{35pt}
	     \caption{Resolutions of Figures \ref{fig:dim_three_vertex_cut_lifted_all_loop_cross_long} to \ref{fig:dim_three_vertex_cut_lifted_cross_cross}. Also in Figures \ref{fig:dim_three_vertex_cut_lifted_cross_cross_long_two_cyc} and \ref{fig:dim_three_vertex_cut_lifted_cross_cross_two_cyc} there is a labeling within the lifted vertex cut such that any self-intersection separate into independent cycles.}
	     \label{fig:dim_three_vertex_cut_lifted_full_implied_cycles_2_resolved}
	     \end{figure}
	     We conclude the entire discussion with considerations in Figure \ref{fig:dim_four_vertex_cut_lifted_full_implied_cycles}.
	     \begin{figure}[H]
	     \centering
	     \hspace*{35pt}
	     \begin{subfigure}[H]{0.36\textwidth}
	         \centering
	         \includegraphics[width=\textwidth]{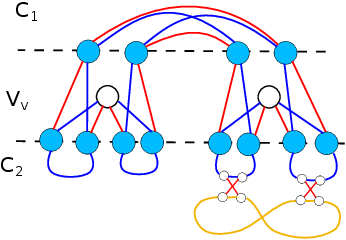}
	         \caption{}
	         \label{fig:dim_four_vertex_cut_lifted_indep_res}
	     \end{subfigure}
	     \hfill
	     \begin{subfigure}[H]{0.36\textwidth}
	         \centering
	         \includegraphics[width=\textwidth]{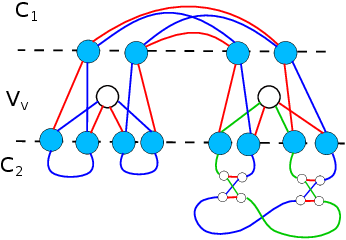}
	         \caption{}
	         \label{fig:dim_four_vertex_cut_lifted_indep_two_cyc}
	     \end{subfigure}
	     \hspace*{35pt}
	     \vskip\baselineskip
	     \hspace*{35pt}
	     \begin{subfigure}[H]{0.36\textwidth}
	         \centering
	         \includegraphics[width=\textwidth]{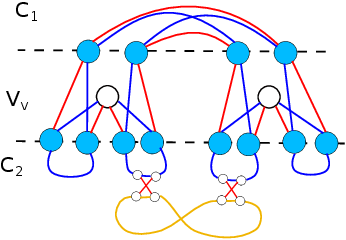}
	         \caption{}
	         \label{fig:dim_four_vertex_cut_lifted_indep_res_1}
	     \end{subfigure}
	     \hfill
	     \begin{subfigure}[H]{0.36\textwidth}
	         \centering
	         \includegraphics[width=\textwidth]{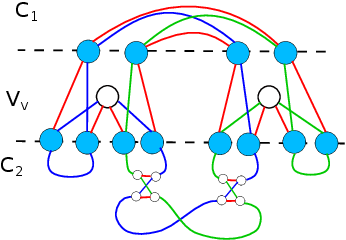}
	         \caption{}
	         \label{fig:dim_four_vertex_cut_lifted_indep_two_cyc_1}
	     \end{subfigure}
	     \hspace*{35pt}
	     \vskip\baselineskip
	     \begin{subfigure}[H]{0.36\textwidth}
	         \centering
	         \includegraphics[width=\textwidth]{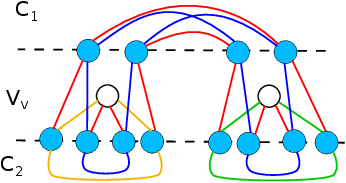}
	         \caption{}
	         \label{fig:dim_four_vertex_cut_lifted_2_2_two_cyc}
	     \end{subfigure}
	     \vskip\baselineskip
	     \hspace*{35pt}
	     \begin{subfigure}[H]{0.36\textwidth}
	         \centering
	         \includegraphics[width=\textwidth]{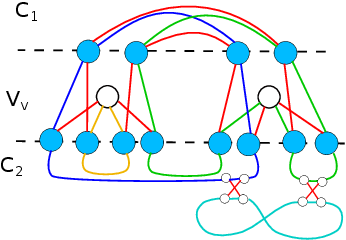}
	         \caption{}
	         \label{fig:dim_four_vertex_cut_lifted_3_1_res_1}
	     \end{subfigure}
	     \hfill
	     \begin{subfigure}[H]{0.36\textwidth}
	         \centering
	         \includegraphics[width=\textwidth]{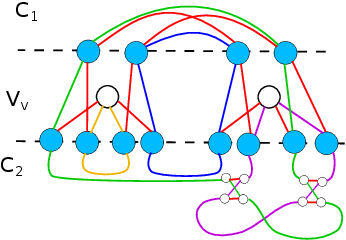}
	         \caption{}
	         \label{fig:dim_four_vertex_cut_lifted_3_1_two_cyc_1}
	     \end{subfigure}
	     \hspace*{35pt}
	     \caption{Resolutions of Figure \ref{fig:dim_four_vertex_cut_lifted_full_implied_cycles}. Note that in Figure \ref{fig:dim_four_vertex_cut_lifted_2_2_two_cyc} any self-intersection within the lifted vertex cut is of Type B independently of the label structure in the lifted vertex cut. Figures \ref{fig:dim_four_vertex_cut_lifted_indep_two_cyc} and \ref{fig:dim_four_vertex_cut_lifted_indep_two_cyc_1} can be completely reduced, i.e., no more Type A self-intersections, by applying the same construction again.}
	     \label{fig:dim_four_vertex_cut_lifted_full_implied_cycles_resolved}
	     \end{figure}
	     This concludes the list of cases which may arise for Type A self-intersections under the conditions of Proposition \ref{prop:self_ints_everywhere_imply_bridge} and the assumption that a bridge exists. In all cases, we can reduce the present Type A self-intersections. Next we show that no are created in the process.
			\subsection{Monotonicity of number of Type A self-intersections}\label{subsub:monotonicity_type_a_deletion}
			So far, we have been occupied with the question of removing Type A intersections. Naturally, we also need to know if and how they might be created, pushing us away from the desired state where no Type A intersections are present. In particular, reducibility of one reduced clique is not sufficient to ensure that no Type A intersections are created by the associated vector of label inversions. We are in particular interested in the cases from Proposition \ref{prop:joining_cycles_in_cubic_graph_double_line_graph} and the resolution of Type A self-intersections as discussed in Subsection \ref{subsub:reduction_Type_A_self_int}. We switch our focus from the lifted vertex cuts to the connecting cycles in $\Gamma_{\Lambda}$, which only appeared in generic form.
			\subsubsection{Choice of connecting cycles through $\mathfrak{G}_{\Lambda}$}\label{sub:conn_cycles_GLambda}
				Given $(\mathfrak{L}_2(G),\mathcal{X},\Lambda)$ and two lifted vertex cuts $\mathbb{X}_v,\mathbb{X}_w\in\mathcal{X}$ resulting in Type A self-intersections, we connected the cycles $\gamma_1,\gamma_2$ containing $\mathbb{X}_v,\mathbb{X}_w$, respectively using Lemma \ref{lem:graph_of_cycles} and Proposition \ref{prop:joining_cycles_in_cubic_graph_double_line_graph}. We can indeed, say more about the possible choices of generic cycles by Lemma \ref{lem:graph_of_cycles}. Consider $\gamma_1,\gamma_2\in\Gamma_{\Lambda}$ as vertices in $\mathfrak{G}_{\Lambda}$. Since $\mathfrak{G}_{\Lambda}$ is simple and connected, there is a shortest path $(\gamma_1,\phi_1,...,\phi_k,\gamma_2)$ in $\mathfrak{G}_{\Lambda}$ such that the vertex induced subgraph of $\mathfrak{G}_{\Lambda}$ is path. For $k=3$ one can visualize the situation in $C_1$ as represented in Figure \ref{fig:chain_cycles_to_connect}. This can be canonically extended to any $k$. The generic view in Figure \ref{fig:chain_cycles_to_connect} is justified due to a planar layout and the fact that the internal structure of the connecting cycles, i.e., Type A self-intersections, are not relevant at this point. They cannot impact the monotonicity because they are already in the "bad" state. 
			\begin{figure}[H]
			     \centering
				 \includegraphics[width=0.45\textwidth]{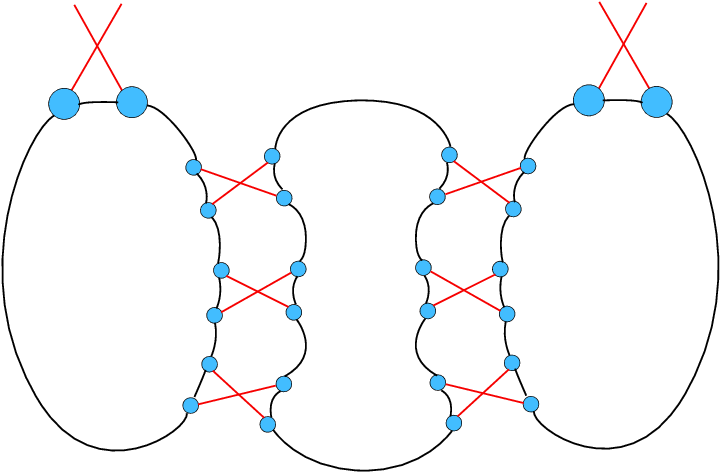}
			     \caption{Generic view of a connecting family of cycles consisting of $3$ adjacent cycles. The cycles containing the Type A self-intersections $\mathbb{X}_v,\mathbb{X}_w$ are represented by the joining reduced cliques in the upper-left and upper-right of the picture. Here it is assumed that exactly $3$ joining reduced cliques exist between each pair $\phi_i,\phi_{i+1}$ but the number can be arbitrary.}
			     \label{fig:chain_cycles_to_connect}
			\end{figure}
			A connecting cycle can then be found by inverting the labels of exactly one joining reduced clique between any two cycles. Note that we can assume that only joining reduced cliques cliques exist between $\phi_i,\phi_{i+1}$ due to the shortest path argument in $\mathfrak{G}_{\Lambda}$. The resulting situation is depicted in Figure \ref{fig:chain_cycles_to_connect_res} for $k=3$, assuming two resolved Type A self-intersections. Remark that any other joining reduced clique, which was not subject to a label inversion, stays a joining reduced clique, but now between two different cycles. No Type A self-intersection is created by inverting the labels of exactly one joining reduced clique between any $\phi_i,\phi_{i+1}$.
			\begin{figure}[H]
			     \centering
				 \includegraphics[width=0.45\textwidth]{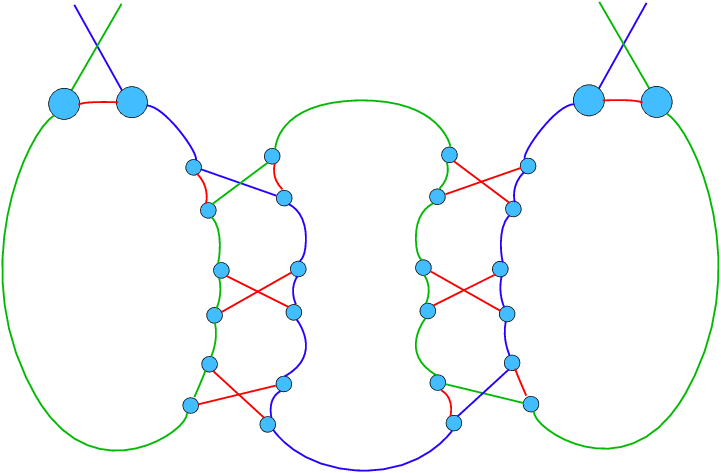}
			     \caption{Cycle configuration after resolution of $\mathbb{X}_v,\mathbb{X}_w$ in line with Figures \ref{fig:dim_two_vertex_cut_lifted_sum_1_1_resolved} to \ref{fig:dim_four_vertex_cut_lifted_full_implied_cycles_resolved} and label inversion of exactly one joining reduced clique between $\phi_i,\phi_{i+1}$. Leads to situation presented for the resolved cases in Figures \ref{fig:dim_two_vertex_cut_lifted_sum_1_1_resolved} to \ref{fig:dim_four_vertex_cut_lifted_full_implied_cycles_resolved} without increasing number of Type A self-intersections.}
			     \label{fig:chain_cycles_to_connect_res}
			\end{figure}
			Inductively in $k$, based on the case $k=3$ from Figure \ref{fig:chain_cycles_to_connect_res}, adding additional $\phi$ to the path $(\gamma_1,\phi_1,...,\phi_k,\gamma_2)$ does not present any problems if one proceeds to invert the labels of exactly one joining reduced clique between any pair $\phi_i,\phi_{i+1}$. This concludes the correct choice of label inversions from Proposition \ref{prop:joining_cycles_in_cubic_graph_double_line_graph} to avoid creation of new Type A self-intersections between the connecting cycles. It remains now to show that also between $\gamma_1$ or $\gamma_2$ and the connecting cycle, new Type A self-intersections can always be avoided.
			\subsubsection{Interaction of Type B intersections with generic connnecting cycles}\label{sub:typeB_generic_cycle}
				The second case, which we have to take into account, is the creation of Type A self-intersections due to Type B self-interactions between the connecting cycles and the cycles containing the lifted vertex cuts. We illustrate the situation based on Figure \ref{fig:type_B_type_A_conversion}. 
			\begin{figure}[H]
			     \centering
				 \includegraphics[width=0.45\textwidth]{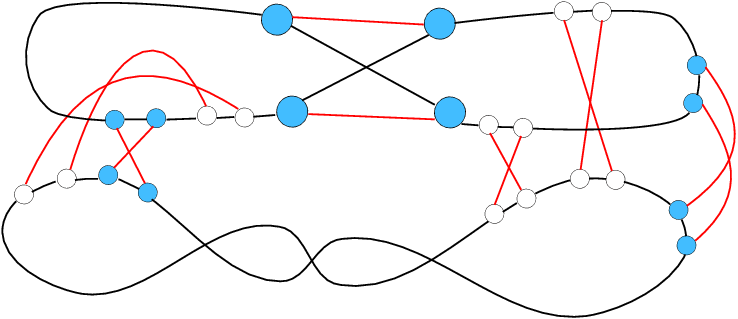}
			     \caption{Generic view of a Type A self-intersection where a cycle connects the "upper" and "lower" part of the cycle relative to the Type A self-intersection. Blue node reduced cliques represent the joining reduced cliques used in Figures \ref{fig:dim_two_vertex_cut_lifted_sum_1_1_resolved}-\ref{fig:dim_four_vertex_cut_lifted_full_implied_cycles_resolved}. }
			     \label{fig:type_B_type_A_conversion}
			\end{figure}
			In Figure \ref{fig:type_B_type_A_conversion_colored} one can see that a connecting cycle and the cycle containing the Type A self-intersection which we want to reduce but which is not transformed via a label inversion $\mathcal{F}_{\mathbb{X}}$. In fact, the potential creation of Type A self-intersection depends fully on the position of the Type B intersections within the final merged cycle. There are three possibilities up to symmetry around the lifted vertex cut, independent of the number of Type B self-intersections joining the cycle $\gamma$ and connecting cycle $\gamma'$. They are shown in Figure \ref{fig:type_B_type_A_conversion_colored} and the corresponding resolution without the creation of additional Type A self-intersections in Figure \ref{fig:type_B_type_A_conversion_connected}. 
			\begin{figure}[H]
			     \centering
				 \includegraphics[width=0.45\textwidth]{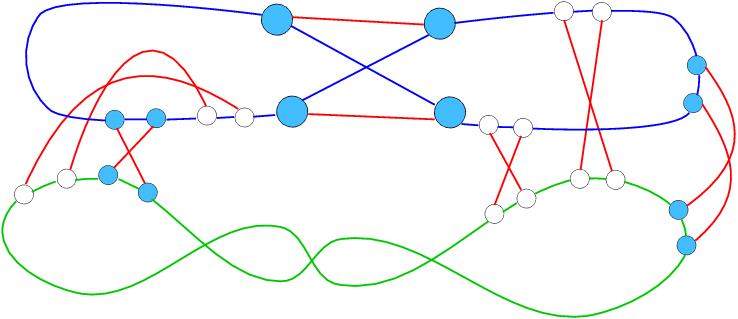}
			     \caption{Type A self-intersection with label inversion and cycle merge giving rise to new Type A self-intersections.}
			     \label{fig:type_B_type_A_conversion_colored}
			\end{figure}
			We want to put the emphasis on the light blue cycle in Figure \ref{fig:type_B_type_A_conversion_connected}, which increases the overall number of cycles and is necessary for the reduction of all self-intersections present in the diagram. Indeed, adding additional joining reduced cliques between $\gamma$ and $\gamma'$ boils down to the same case as presented in Figure \ref{fig:type_B_type_A_conversion_colored}, because all relative pairwise positions of reduced cliques are covered by Figure \ref{fig:type_B_type_A_conversion_colored}. 
			\begin{figure}[H]
			     \centering
				 \includegraphics[width=0.45\textwidth]{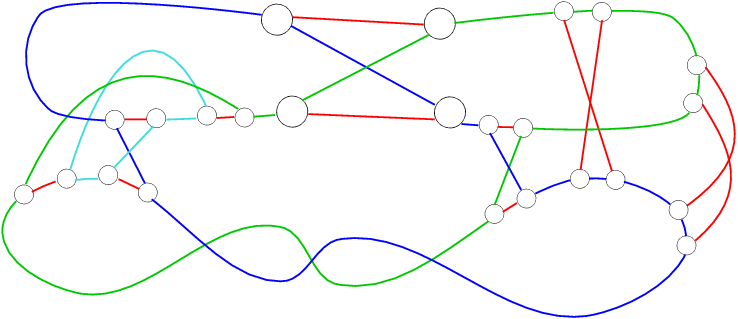}
			     \caption{Type A self-intersection with label inversion and cycle merge giving rise to new Type A self-intersections.}
			     \label{fig:type_B_type_A_conversion_connected}
			\end{figure}
			Therefore, based on the arguments in Subsections \ref{sub:conn_cycles_GLambda} and \ref{sub:typeB_generic_cycle}, we can always avoid creating Type A self-intersections from Type B self-intersections in any case discussed in Subsection \ref{subsub:reduction_Type_A_self_int}. Consequently, putting together Subsections \ref{subsub:reduction_Type_A_self_int} to \ref{subsub:monotonicity_type_a_deletion}, we can conclude under the assumptions of Proposition \ref{prop:self_ints_everywhere_imply_bridge}, that there is for any initial labeled graph-subgraph structure $(\mathfrak{L}_2(G),\mathcal{X},\Lambda)$ a finite sequence $\mathbb{F}=(\mathbb{F}_{l})$ of vectors of label inversions 
			\begin{equation*}
				\mathbb{F}_l = \mathcal{F}_{X_{i_n}}\circ\hdots\circ \mathcal{F}_{X_{i_1}}
			\end{equation*} 
			such that the number of Type A self-intersections in $(\mathfrak{L}_2(G),\mathcal{X},\mathbb{F}_l(\Lambda))$ is strictly monotonously decreasing in $l$, integer-valued and non-negative. This implies, that we find on any bridgeless graph a valid labeling $\Lambda$ such that $(\mathfrak{L}_2(G),\mathcal{X},\Lambda)$ has no Type A self-intersections. \par
		This concludes the proof of  Proposition \ref{prop:self_ints_everywhere_imply_bridge} allowing us to tackle proof of the Cycle Double Cover conjecture.
\section{Proof of Theorem \ref{thm:cycle_double_cover}}\label{sec:proof_main_thm}
	The proof of Theorem \ref{thm:cycle_double_cover} is fully based on the results concerning the reduced order two line graph $\mathfrak{L}_2(G)$  with an appropriate binary open-close labeling of the edges. The proof then works as presented in Figure \ref{fig:projection_diagram}.
	\begin{figure}[H]
		\centering
		\begin{tikzcd}
			G  \arrow[r, "\mathcal{L}"] & \mathcal{L}(G) \arrow[r, "\mathcal{L}"] \arrow[d] & \mathcal{L}(\mathcal{L}(G)) \arrow[r] & (\mathfrak{L}_2(G),\mathcal{X},\Lambda) \arrow[d, "\mathbb{T}_{\Lambda}"]  \\
			& \mathcal{L}(G) \arrow[lu, bend left, "\pi"] & & \Gamma_{\Lambda,\mathrm{decomp}} \arrow[ll, "\chi_{\mathcal{E}}"] \\\
		\end{tikzcd}
		\caption{The proof setup with objects and construction of $(\mathfrak{L}_2(G),\mathcal{X},\Lambda)$ (Definition \ref{def:open_closed_edges_in_reduced_double_line}), $\mathbb{T}_{\Lambda}$ set of label inversions (Definition \ref{def:label_inversions}) encapsulating the construction from Section \ref{sec:construction_sec_order_line_graph_red}.}
		\label{fig:projection_diagram}
	\end{figure}	 
	\begin{proof}[Proof of Theorem \ref{thm:cycle_double_cover}]
		Under the assumption that $G$ is bridgeless, we obtain from Proposition \ref{prop:self_ints_everywhere_imply_bridge} that there is a set of labelings $\Lambda^{\ast}$ of $\mathfrak{L}_2(G)$ such that the resulting cycles $\Gamma_{\Lambda^{\ast}}$ do not contain any Type A intersections. Consequently, all intersections are of Type B and therefore reducible. By Lemma \ref{lem:intersections_joining_cycles_after_flips} reducing the Type B intersections gives a monotone process until no more intersections are present and we call the resulting set of intersection-free cycles $\Gamma_{\Lambda,\mathrm{decomp}}$ and the orientation simply $\Lambda$. Applying $\chi_{\mathcal{E}}$ to $\Gamma_{\Lambda,\mathrm{decomp}}$, we obtain a coloring of $\mathcal{L}(G)$ such that exactly two edges incident to a vertex $q\in  \mathcal{L}(G)$ have the same color and both do not belong to the same triangle induced by a vertex in $G$. As depicted in Figure \ref{fig:projection_line_graph_colors_double_cover} we obtain a Cycle Double Cover of $G$ by projection of $\Lambda$ via a map $\pi$ which assigns any edge in $G$ the two colors of the edges which are incident to the correspond vertex in $\mathcal{L}(G)$ and call it $\mathcal{C}$.
		\begin{figure}[H]
			\centering
			\includegraphics{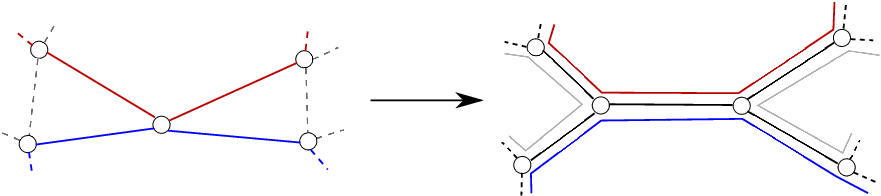}
			\caption{Projection from line graph coloring to cycle cover of $G$ satisfying the conditions of Theorem \ref{thm:cycle_double_cover}. The arrow indicates the mapping through $\pi$.}	
			\label{fig:projection_line_graph_colors_double_cover}
		\end{figure}
		As in the proof of Lemma \ref{lem:projection_without_intersection_results_in_cycles} if the set of colored edges $\mathcal{C}$ defined a closed walk, the inverse operation would imply that there is a vertex in $\mathcal{L}(G)$ such that all incident edges have the same color. This is impossible by construction of $\Lambda^{\ast}$ such that all edges in $G$ are traversed by exactly two cycles from $\pi(\Lambda)$.
	\end{proof}

\section{Outlook}
	The motivation for the whole construction used in the proof of Proposition \ref{prop:self_ints_everywhere_imply_bridge} comes from a connection between the notion of left hand paths and vertex orientations as discussed in \cite{Br04} combined with a understanding of cubic graphs with assigned vertex orientations as a spin system, since vertex orientations on cubic graphs are binary in their nature. One could, therefore, ask the question, whether a Cycle Double Cover of $G$ can be seen as the equilibrium of some spin system relative to some Hamiltonian and the whole process of removing Type A self-intersections could be understood as an equivalent to Glauber dynamics. As we have seen in Subsections \ref{subsub:reduction_Type_A_self_int} and \ref{subsub:monotonicity_type_a_deletion}, the necessary vectors of label inversions to reduce a Type A self-intersection are highly non-local in the sense that they not only impact some reduced clique and its neighbors but might be influencing reduced cliques very "far away" from the Type A self-intersection, which we want to reduce. The correct view is on the level of cycles in $\mathfrak{L}_2(G)$ and a configuration $\Gamma_{\Lambda}$ would be the set of cycles induced by valid labels $\Lambda$. This allows us to formulate the Hamiltonian as
	\begin{equation*}
		\mathcal{H}(\Gamma_{\Lambda}) = \sum_{\gamma_i,\gamma_j\in \Gamma_{\Lambda}}\Phi(\gamma_i,\gamma_j)
	\end{equation*}
	where $\Phi$ is some symmetric interaction. This interaction $\Phi$ should be constructed from the cases discussed in Subsection \ref{subsub:reduction_Type_A_self_int}. This results in a possibility of sampling Cycle Double Covers, which are computationally still not approachable since the construction from Subsection \ref{subsub:reduction_Type_A_self_int} is not algorithmically feasible in the sense, that it needs non-polynomial effort in the size of the graph.

\newpage 

\appendix
\section{Centered vertex-cuts in $\mathcal{L}(G)$ and their lifts}\label{app:centered_vertex_cut}
	The central technique for the proof of Proposition \ref{prop:self_ints_everywhere_imply_bridge} were centered vertex cuts, build around all possible Type A self-intersections. Now we need to discuss and prove the topological result, which underlined the proof. A $v$ centered vertex cut $\mathbb{V}_v$ of $\mathcal{L}(G)$ with a bridge-free underlying cubic-graph $G$ satisfies necessarily $|\mathbb{V}_v|\in\lbrace 2,3,4\rbrace$, since $\mathcal{L}(G)$ is $4$ regular and $|\mathbb{V}_v|=1$ would imply a bridge. 
		\begin{figure}[H]
	     \centering
	     \begin{subfigure}[H]{0.45\textwidth}
	         \centering
	         \includegraphics[width=\textwidth]{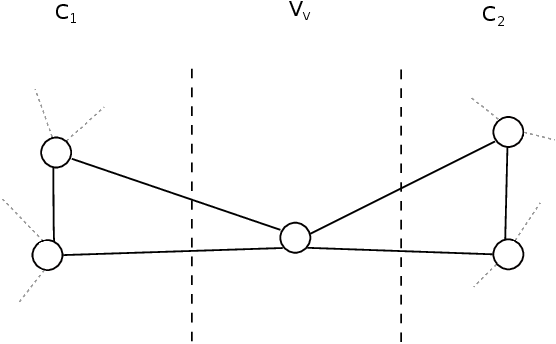}
	         \caption{A single vertex in the vertex cut with all its neighbors in one connected component or the other.}
	     \end{subfigure}
	     \hfill
	     \begin{subfigure}[H]{0.36\textwidth}
	         \centering
	         \includegraphics[width=\textwidth]{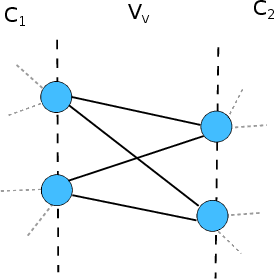}
	         \caption{Lift of a balanced vertex.}
	         \label{fig:balanced_cut_vertex_lifted}
	     \end{subfigure}
	     \caption{Cut vertex in $\mathbb{V}_v$ belonging to two triangles each in one connected component and its lift.}
	\end{figure}
	We will consider in what follows all three cases. In each case there is one local structure which is added and the remaining cases are combinatorial combinations of previous cases. We call these combinatorial combinations the direct sum of any number of lower order cases.\par
	\textit{Case 1 with $|\mathbb{V}_v|=2$:} 
		\begin{figure}[H]
	     \centering
	     \begin{subfigure}[H]{0.28\textwidth}
	         \centering
	         \includegraphics[width=\textwidth]{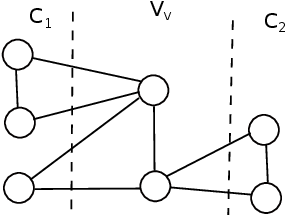}
	         \caption{The two vertices in $\mathbb{V}_v$ being neighbors.}
	         \label{fig:dim_two_vertex_cut}
	     \end{subfigure}
	     \hfill
	     \begin{subfigure}[H]{0.28\textwidth}
	         \centering
	         \includegraphics[width=\textwidth]{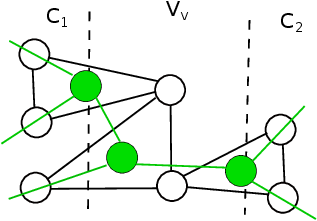}
	         \caption{This configuration in the line graph $\mathcal{L}(G)$ implies a bridge in $G$ for $|\mathbb{V}_v|=2$.}
	         \label{fig:dim_two_vertex_cut_implies_bridge}
	     \end{subfigure}
	     \hfill
	     \begin{subfigure}[H]{0.28\textwidth}
	         \centering
	         \includegraphics[width=\textwidth]{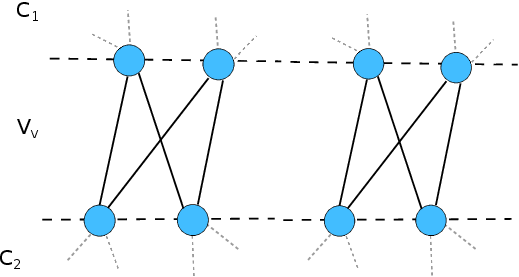}
	         \caption{Direct sum of lift of two balanced vertices.}
	         \label{fig:balanced_cut_vertex_lifted_direct_sum}
	     \end{subfigure}
	     \caption{All possible situations for a cut vertex in $\mathbb{V}_v$ in $\mathcal{L}(G)$. A full triangle cannot lie within $\mathbb{V}_v$.}
		\label{fig:dim_two_vertex_cut_bridge}
	\end{figure} 
	From Figure \ref{fig:dim_two_vertex_cut_bridge} we can see, that certain configurations for the centered vertex cut $\mathbb{V}_v$ are outside of conditions of the theorem, which we try to prove. In particular, whenever only the neighbors of one single vertex in $\mathbb{V}_v$ lie in $C_1$ or $C_2$, this vertex becomes a cut-vertex for $\mathcal{L}(G)$, as can be seen in Figure \ref{fig:dim_two_vertex_cut_implies_bridge}. Since $G$ is assumed to be bridgeless, this cannot happen. But, nonetheless, the configuration may occur as a building block for higher order cases of $|\mathbb{V}_v|=3$. They can then be resolved as the direct sum of two balanced vertices as shown in Figure \ref{fig:balanced_cut_vertex_lifted_direct_sum}.
	\par
		\textit{Case 2 with $|\mathbb{V}_v|=3$:}
		\begin{figure}[H]
	     \centering
	     \hspace*{35pt}
	     \begin{subfigure}[H]{0.36\textwidth}
	         \centering
	         \includegraphics[width=\textwidth]{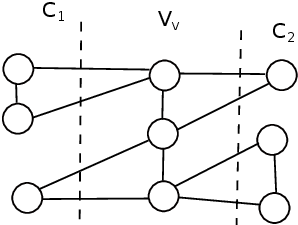}
	         \caption{A balanced vertex in $\mathbb{V}_v$ with all its neighbors in the connected components.}
	         \label{fig:dim_three_vertex_cut}
	     \end{subfigure}
	     \hfill
	     \begin{subfigure}[H]{0.28\textwidth}
	         \centering
	         \includegraphics[width=\textwidth]{images/dim_three_vertex_cut_lifted}
	         \caption{A semi-balanced vertex in $\mathbb{V}_v$ with one neighbor also in $\mathbb{V}_v$.}
	         \label{fig:dim_three_vertex_cut_lifted}
	     \end{subfigure}
		\hspace*{35pt}
	     \caption{All possible situations for a cut vertex in $\mathbb{V}_v$ in $\mathcal{L}(G)$. A full triangle cannot lie within $\mathbb{V}_v$.}
		\label{fig:dim_three_vertex_cut_cases}
	\end{figure} 
	From Figure \ref{fig:dim_three_vertex_cut_cases} we obtain further building blocks for the proof of Proposition \ref{prop:self_ints_everywhere_imply_bridge}. 
		\par
		\textit{Case 3 with $|\mathbb{V}_v|=4$:} In the case that all three preceding cases do not apply to the $v$ centered vertex cut $\mathbb{V}_v$, we simply chose a neighbor $w_v$ of $v$ and $\mathbb{V}_v=\mathcal{N}_{w_v}$. This leads to a an "empty" lift of the connected components $\lbrace w_v\rbrace$. Nonetheless, this suffices to make the conclusions about resolutions of Type A self-intersections.
		\begin{figure}[H]
	     \centering
	     \begin{subfigure}[H]{0.36\textwidth}
	         \centering
			\hspace*{35pt}
	         \includegraphics[width=\textwidth]{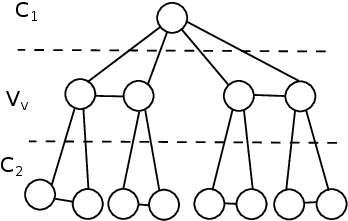}
	         \caption{A balanced vertex in $\mathbb{V}_v$ with all its neighbors in the connected components.}
	         \label{fig:balanced_cut_vertex_situation}
	     \end{subfigure}
	     \hfill
	     \begin{subfigure}[H]{0.36\textwidth}
	         \centering
	         \includegraphics[width=\textwidth]{images/dim_four_vertex_cut_lifted}
	         \caption{A semi-balanced vertex in $\mathbb{V}_v$ with one neighbor also in $\mathbb{V}_v$.}
	         \label{fig:semibalanced_cut_vertex_situation}
	     \end{subfigure}
		\hspace*{35pt}
	     \caption{All possible situations for a cut vertex in $\mathbb{V}_v$ in $\mathcal{L}(G)$. A full triangle cannot lie within $\mathbb{V}_v$.}
		\label{fig:balanced_unbalanced_cut_vertex_cases}
	\end{figure} 
%	\begin{figure}
%         \centering
%         \includegraphics[width=\textwidth]{images/dim_four_vertex_cut_lifted_resolving}
%         \caption{One possible resolution of Type A self-intersection $\mathbb{V}_v$ with $|\mathbb{V}_v|=4$. Illustrated is the case with three cycles.}
%         \label{fig:semibalanced_cut_vertex_situation}
%	\end{figure}
	\par
	Note that at most two neighbors of each $v'\in\mathbb{V}_v$ can belong to any connected component and the remaining ones have to be part of the vertex cut $\mathbb{V}_v$. Otherwise, due to the triangular structure of $\mathcal{L}(G)$, an edge would connect some $C_i$ and $C_j$, which is a contradiction to $\mathbb{V}_v$ being a vertex cut. Furthermore, no triangle in $\mathcal{L}(G)$ belongs completely to $\mathbb{V}_v$ due to the minimality condition. If we considered a vertex cut, which contains a triangle, then we can move one of the corners of the triangle to one of its adjacent connected components, i.e., in which it has a neighbor. Remember that this corner cannot have neighbors in two different connected components. This reduces the size of the vertex cut by $1$ while still being a vertex cut. This is a contradiction to minimality.\par
	Finally, since $G$ is bridgeless, between any two $C_i,C_j$ with $i\neq j$ there are at least two vertices $v',w'\in\mathbb{V}_v$ adjacent to vertices in $C_i$ and $C_j$. We only consider the connected components adjacent to $v$ of which there are exactly two and they are uniquely defined by our choice of $v$. We may, hence, assume without loss of generality, that there are exactly two connected components after deletion of any $v$-centered vertex cut $\mathbb{V}_v$. All types of cut vertices can be lifted to $(\mathfrak{L}_2(G),\mathcal{X})$ and take then the form of reduced cliques.
\section{A construction using half edges}\label{app:half_edge_alternative}
	Even though the construction of $\mathfrak{L}_2(G)$ seems natural from a perspective of line graphs and decoupling cycles by edge properties, it requires an abstraction from the underlying graph $G$, which might seem far fetched. It turns out, that it is possible to introduce all necessary structure on the level of $G$ by replacing its neighborhood properties by employing an extension of half-edges. To this end, consider the graph $G=(V,E)$ as in the conditions of Theorem \ref{thm:cycle_double_cover}. We replace the edge set $E$ by the following structure. For each $v\in V$, consider three pairs of half-edges $(e_{v}^{(1,1)},e_{v}^{(1,2)})$, $(e_{v}^{(2,1)},e_{v}^{(2,2)})$ and $(e_{v}^{(3,1)},e_{v}^{(3,2)})$ with source $v$. Each pair is represented in Figure \ref{fig:half_edge_pairs_no_or} by a line with identical end decorations. Each pair will belong to the same walk or cycle in a double cover. If $\langle v,w\rangle \in E$ for some $v,w\in V$ then, from four pairs of half edges $(e_{v}^{(i_1,1)},e_{v}^{(i_1,2)})$, $(e_{v}^{(i_2,1)},e_{v}^{(i_2,2)})$ and $(e_{w}^{(j_1,1)},e_{w}^{(j_1,2)})$, $(e_{v}^{(j_2,1)},e_{v}^{(j_2,2)})$ we construct a crossing consisting of four connections, as represented in Figure \ref{fig:half_edge_pairs_no_or}. 
	\begin{figure}[H]
		\centering
         \includegraphics[width=0.7\textwidth]{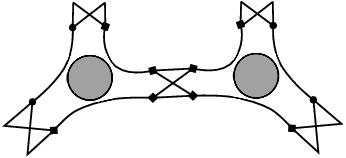}
         \caption{Half-edge pairs for double walk covers and crossings between vertices, connecting different walks.}
         \label{fig:half_edge_pairs_no_or}
     \end{figure}
     Note that contraction of all half-edge pairs around the same vertex, represented in Figure \ref{fig:half_edge_pairs_no_or} by identical end points, leads to exactly the same structure $\mathfrak{L}_2(G)$ as was used throughout this article. Additionally, the set of reduced cliques $\mathcal{X}$ corresponds to the crossings between walks. Consequently, we can again introduce labels on the crossings for open-closed pairs, making both views equivalent. It was the authors choice to go with the line graph approach due to being more versed in this perspective on graphs. Nonetheless, it turns out that for a generalization of the proof to the oriented double cycle cover conjecture, the half-edge construction would, indeed, be more fruitful. Using the notion of left hand paths and vertex orientations as discussed in \cite{Br04} one can go from Figure \ref{fig:half_edge_pairs_no_or} to Figure  \ref{fig:half_edge_pairs_w_or} by introducing vertex orientations in form of the given arrows.
	\begin{figure}[H]
		\centering
         \includegraphics[width=0.7\textwidth]{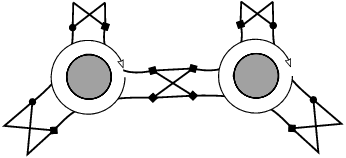}
         \caption{Introducing vertex orientations, indicated by arrows around each vertex, together with half-edge construction gives potential approach to oriented double cycle cover.}
         \label{fig:half_edge_pairs_w_or}
     \end{figure}
     Now, there is an interplay between the orientations on the vertices and the open-closed pairs in the crossings between pairs. Changing one implies the automatically changes on the other and the author has not yet managed to find a way to resolve this dependence. It seems, nonetheless, like a good starting point for the search for oriented double cycle covers and if it is helpful to someone else in answering that question, the author would be extremely happy about it.

\section*{Statements and Declarations}
This work was done independently after having finished my PhD in June 2022. \\
The author has no relevant financial or non-financial interests to disclose. The article was exclusively written by the mentioned author.

\printbibliography

\end{document}